\documentclass[11pt, reqno]{amsart}

\usepackage{amssymb}
\usepackage{graphicx}
\usepackage{amsfonts}
\usepackage{amsthm}
\usepackage{amsmath}
\usepackage{empheq}
\usepackage{algorithm}
\usepackage{algorithmic}
\usepackage[T1]{fontenc}
\usepackage[usenames,svgnames,dvipsnames]{xcolor}
\usepackage{bbm}

\usepackage{subcaption}
\usepackage{epsfig}
\usepackage{epstopdf}

\usepackage[bookmarksnumbered=true, citecolor=blue, colorlinks=true]{hyperref}
\usepackage[nameinlink]{cleveref}
\crefname{equation}{}{} % suppress eq. in cref equations

\usepackage[numbers,sort&compress]{natbib}
\usepackage{todonotes}
\usepackage{mathtools}

\theoremstyle{plain}
\newtheorem{theorem}{Theorem}[section]
\newtheorem{lemma}[theorem]{Lemma}

\newtheorem{assumption}[theorem]{Assumption}
\theoremstyle{definition}

\theoremstyle{remark}

\numberwithin{equation}{section}

\newcommand{\bN}{\mathbb N}
\newcommand{\bR}{\mathbb R}

\newcommand{\Dl}{{D_{\lambda}}}

\newcommand{\DlR}{{D_{\lambda,R}}}
\newcommand{\DlRc}{{V\backslash D_{\lambda,R}}}

\newcommand{\Gl}{{W_{\lambda, R}}}
\newcommand{\cG}{\mathcal{G}}
\newcommand{\cS}{\mathcal{S}}

\newcommand{\veta}{\boldsymbol \eta}
\newcommand{\vtheta}{\boldsymbol \theta}

\newcommand{\mtrx}[1]{\mathsf{#1}}

\newcommand{\mB}{{\bf B}}

\newcommand{\mG}{{\bf G}}

\newcommand{\mI}{{\bf I}}

\newcommand{\mzero}{\mtrx 0}

\newcommand{\st}{\mbox{ s.t. }}

\begin{document}

\title[Exponential Convergence of Divide-and-Conquer Algorithm]{Exponential convergence of a distributed divide-and-conquer algorithm for constrained convex optimization on networks}

\author{Nazar Emirov}
\address  %{Amadeus North America Inc, Newcastle, WA 02467}
{Department of Mathematics, University of Central Florida, Orlando, Florida 32816}
%\curraddr{Department of Mathematics and Statistics,
%Case Western Reserve University, Cleveland, Ohio 43403}
\email{nazaremirov@gmail.com}

\author{Guohui Song}
\address{Department of Mathematics and Statistics, Old Dominion University, Norfolk, Virginia 23529}
\email{gsong@odu.edu}

\author{Qiyu Sun}
\address{Department of Mathematics, University of Central Florida, Orlando, Florida 32816}
\email{qiyu.sun@ucf.edu}

\thanks{The project is partially supported by the National Science Foundation DMS-1816313 and DMS-2318781.}

\dedicatory{Dedicated to Professor David Royal Larson}

%\keywords{kk}

\begin{abstract}
We propose a divide-and-conquer (DAC) algorithm for constrained convex optimization over networks, where the global objective is the sum of local objectives attached to individual agents. The algorithm is fully distributed: each iteration solves local subproblems around selected fusion centers and coordinates only with neighboring fusion centers. 
Under standard assumptions of smoothness, strong convexity, and locality on the objective function, together with polynomial growth conditions on the underlying graph, we establish exponential convergence of the DAC iterations and derive explicit bounds for both exact and inexact local solvers.
Numerical experiments on three representative losses ($L_2$ distance, quadratic, and entropy) confirm the theory and demonstrate scalability and effectiveness.
\end{abstract}

\maketitle

\section{Introduction}

In this paper, we address the solution of the following constrained convex optimization problem on a network described by a simple graph ${\mathcal G}=(V, E)$ of large order $N$, 
\begin{subequations} \label{convexoptimizationgeneral.def}
  \begin{equation}\label{convexoptimizationgeneral.def1}
    \min_{{\bf x}\in  {\mathbb R}^N} F({\bf x}):=\sum_{i\in V}f_i({\bf x}),
  \end{equation}
subject to linear equality constraints on vertices in   $W\subseteq V$,
\begin{equation} \label{convexoptimizationgeneral.def2}
    {\bf A} {\bf x} = {\bf b},
  \end{equation}
  and convex inequality constraints on vertices in $U\subseteq V$,
  \begin{equation} \label{convexoptimizationgeneral.def3}
    g_l({\bf x})\le 0, \ l\in U,
  \end{equation}
\end{subequations}
where the local objective functions $f_i, i\in V$, and constraint functions $g_l, l\in U$, are smooth and depend only on neighboring variables, and the constraint matrix ${\bf A} = [a(k,j)]_{k\in W, j\in V}$ has a small bandwidth; see \Cref{preliminaries.section} for a detailed description of the problem setup.  This paper extends our earlier work \cite{Emirov2022}, in which a distributed divide-and-conquer algorithm is proposed to handle an unconstrained convex optimization problem.

The constrained convex optimization framework \eqref{convexoptimizationgeneral.def}, in which the global objective $F$ is composed of local objectives $f_i$ associated with individual agents $i$ in the network, has extensive applications across a broad spectrum of fields, including distributed machine learning, environmental monitoring, smart grids and energy systems, and predictive control of models. It has received significant attention from researchers working in these and related fields; see \cite{Yang2019a, Bertsekas2015, Cao2013, Nedich2015, Yang2010} and  references therein. Many algorithms, such as the augmented Lagrangian method (ALM),  the alternating direction method of multipliers (ADMM), and the primal-dual subgradient method, have been proposed to solve the above constrained convex optimization problem with strong empirical performance and some theoretical guarantees \cite{Boyd11, Carli2020, Falsone2017, Chang2014, Liang2020, Nedic2009, Shi2015, Nedic2017}.  However, the practical implementation of these algorithms may be infeasible in many large-scale applications due to significant communication overhead, computational limitations, and associated costs.
In this paper, we propose a fully distributed algorithm to solve the constrained convex optimization problem \eqref{convexoptimizationgeneral.def} on networks. Our divide-and-conquer (DAC) approach consists of three key steps:
\begin{enumerate}
\item \emph{Decomposition}: The original constrained convex optimization problem is partitioned into a set of interactive subproblems, each centered around a fusion center located at a vertex in $\Lambda\subseteq V$;

\item \emph{Local Computation}:  These subproblems are solved independently using the computational resources available at the respective fusion centers;

\item \emph{Coordination and Update}:  The solutions are propagated and updated across neighboring fusion centers to ensure global coordination and  exponential convergence throughout the network.
\end{enumerate}

In our distributed and decentralized implementation of the proposed DAC algorithm, all data storage, exchange, and numerical computation are performed at fusion centers located at the vertices in $\Lambda\subseteq V$; see Algorithm \ref{Main.algorithm}. The selection of fusion centers can be tailored to specific application requirements. For instance, choosing a single fusion center results in a centralized implementation, whereas selecting all vertices as fusion centers yields a fully decentralized system, resembling autonomous robotic control systems operating without centralized monitoring.

The interactive subproblems are substantially smaller than the original optimization problem and can be solved more efficiently under resource constraints, as each subproblem depends  on the neighboring state variables associated with its fusion center, and then consequently the computational domain of each fusion center aligns approximately with its local neighborhood; see \eqref{DACwithoutinequality.defb} and \Cref{fusioncenter.subsection}.

At each update step of the proposed DAC algorithm, each fusion center exchanges the solution of its local subproblem only with a small number of neighboring fusion centers. This localized interaction significantly reduces communication overhead and enables the DAC approach to scale efficiently with both the size and structural complexity of the network.

In contrast to ALM and ADMM, our approach does not require the construction of a global solution at each iteration. This key feature enables the DAC algorithm to operate in a fully distributed and decentralized manner, significantly reducing communication overhead and enhancing the scalability. Therefore, the proposed method is well-suited for large-scale network systems where centralized coordination is impractical or costly.

This paper is organized as follows. \Cref{preliminaries.section} outlines the
assumptions regarding the underlying graph $\cG=(V, E)$, the objective function $F=\sum_{i\in V} f_i$, the inequality constraint functions $g_l, l\in U\subseteq V$, and  the linear constraint matrix ${\bf A} %=[a(i,j)]_{i,j\in V}
$ associated with the constrained convex optimization problem \eqref{convexoptimizationgeneral.def}.
This section also details the memory, communication, and computational requirements at the fusion centers necessary for implementing the proposed DAC algorithm.
\Cref{dac.section} introduces the DAC algorithm for solving the constrained optimization problem \eqref{convexoptimizationgeneral.def} and provides a convergence analysis; see Algorithm \ref{Main.algorithm} and Theorems \ref{maintheorem1.thm}, \ref{inexactmaintheorem1.thm} and \ref{thm:logbarrier}.   Section \ref{numericalexperiments.section} presents numerical experiments that illustrate the performance and effectiveness of the DAC algorithm. Finally, Section \ref{proofs.section} contains the proofs of the main theoretical results.

\section{Problem setting}\label{preliminaries.section}

In this follow-up to our earlier work \cite{Emirov2022}, we use the same setting on  the underlying graph  ${\mathcal G}:=(V, E)$,
the objective function $F$, and the fusion center $\Lambda\subseteq V$ for distributed and decentralized implementation of the proposed divide-and-conquer algorithm. 
In addition, we assume that the constrained functions $g_l, l\in U\subseteq V$, are local, smooth and convex, and the constrained matrix ${\bf A}$ are local and stable.

\subsection{Graphs with polynomial growth property}\label{polynomialgrowth.subsection}
In this paper, we consider the  constrained convex optimization problem \eqref{convexoptimizationgeneral.def}
with 
the underlying graph  ${\mathcal G}$ being finite, simple and connected,
and  satisfying the following  {\em polynomial growth property}: 

\begin{assumption}\label{PolynomialGrowth.assumption}
 {\em  There exist positive constants  $d(\cG)$ and $D_1(\cG)$  such that
  \begin{equation}\label{PolynomialGrowth}
    \mu_{\cG}(B(i,R))\le D_1(\cG)(R+1)^{d(\cG)}\ \ \text{for all}\  i\in V \text{ and } \ R\geq 0.
  \end{equation}
Here the   counting measure $\mu_\cG$ on the graph  $\cG$  assigns each vertex subset $V_1\subseteq V$
its cardinality,    the   geodesic distance   $\rho(i, j)$ between vertices $i$ and $j\in V$ is the number of edges in the shortest path to connect them, 
and the  closed ball
  $$B(i,R):=\big\{j\in V, \ \rho(i,j)\le R\big\}, \ R\ge 0$$
represents  the collection of
 all $R$-neighbors of a vertex $i\in V$.
}\end{assumption}

The minimal positive constants $d(\cG) $ and $D_1(\cG)$ in the polynomial growth property \eqref{PolynomialGrowth} are known as {\em Beurling dimension} and {\em density} of the graph $\cG$ respectively \cite{Cheng2019, Shin19, Yang2013}.

In this paper, we denote the order  of the underlying graph ${\mathcal G}$ by $N = \mu_{\mathcal G}(V)$, and write $V=\{1, \ldots, N\}$ for simplicity.

\subsection{Objective functions}\label{objectivefunction.subsection}

In this paper, we require that
local objective functions
$f_i, i\in V$, in the decomposition  of   the  global objective function 
  \begin{equation}\label{objectivefunction.decomposition}
 F({\bf x})=\sum_{i\in V} f_i({\bf x}),
 \end{equation}
in  the 
 constrained optimization \eqref{convexoptimizationgeneral.def}  
 are  smooth and  centered around the vertices of the underlying graph ${\mathcal G}$.

 \begin{assumption}\label{assump:f}
The local  objective functions $f_i, i\in V$, in \eqref{objectivefunction.decomposition}
 are continuously differentiable and depend only on neighboring variables  $x_{j}, j\in B(i, m)$,  
 where $m\ge 1$ 
and ${\bf x}=[x_j]_{j\in V}$ with $x_j\in {\mathbb R}, j\in V$.
\end{assumption}
 
 In this paper, we call the minimal integer $m$ in Assumption  \ref{assump:f} as the {\em neighboring radius} 
 of  the local objective functions $f_i, i\in V$, and we use it 
  to describe their localization.  For the special case that $m=0$, 
  %the local function $f_i$ depends only on $x_i, i\in V$ and hence 
 the objective function $F$ is  separable, i.e., 
  $$F(x)=\sum_{i\in V} g_i(x_i)$$
  for some continuously differentiable functions $g_i, i\in V$, on the real line.

Let $V_1, V_2\subseteq V$.
For a  block matrix ${\bf B}=[b(i,j)]_{i\in V_1,j\in V_2}$  with matrix  entries $b(i,j)\in {\mathbb R}$ for $i\in V_1$ and $j\in V_2$, we define its
 {\em geodesic-width}
by
$$\omega({\bf B})=\min \big\{ \omega: \
  \
  b(i,j)=0 \ \ {\rm  for\ all} \ \  i,j\in V \ \ {\rm with} \  \  \rho(i,j)>\omega\big\}.$$
Therefore the neighboring requirements of the local objective functions $f_i, i\in V$,
 in  Assumption \ref{assump:f}
 can be described as
\begin{equation}\label{assumption2.characterization}
  \omega\Big( \Big [\frac{\partial f_i({\bf x})}{\partial x_j}\Big ]_{i, j\in V}\Big)\le m 
\end{equation}
where ${\bf x}=[x_j]_{j\in V}\in {\mathbb R}^N$.

\smallskip

In this paper,  we use ${\bf A} \preceq {\bf B}$ to denote that ${\bf B}-{\bf A}$ is positive semi-definite.
In addition to the neighboring requirement in Assumption   \ref{assump:f} for the local objective functions $f_i, i\in V$,
%in the decomposition \eqref{objectivefunction.decomposition}, 
we also require the global objective function $F$ in  the
 constrained optimization \eqref{convexoptimizationgeneral.def}
 is smooth and strongly convex.

\begin{assumption}\label{assump:J}
  There exist  positive constants $0<c_1<L_1<\infty$ and  positive definite matrices ${\bf J}({\bf x}, {\bf y})$  %, {\bf x}, {\bf y}\in {\mathbb R}^{S}$,
  such that
  \begin{subequations}\label{assump:J.eq1}
    \begin{equation}
      \omega({\bf J}({\bf x},{\bf y}))\le 2m, \quad c_1\mI \preceq {\bf J}({\bf x}, {\bf y}) \preceq L_1\mI,
    \end{equation}
    and
    \begin{equation}\label{eq:gradient_difference}
      \nabla F({\bf x}) - \nabla F({\bf y}) = {\bf J}({\bf x}, {\bf y})({\bf x} -{\bf y})
    \end{equation}
  \end{subequations}
  hold for all ${\bf x},{\bf y}\in {\mathbb R}^N$, where $m$ is the neighboring radius of the local objective functions $f_i, i\in V$. 
\end{assumption}

We remark that Assumption  \ref{assump:J} holds for the global objective function $F$  when it
satisfies the conventional strict convexity condition,
$$c_1 {\mI} \preceq \nabla^2 F({\bf x})\preceq L_1 {\mI}\ \ {\rm for \ all} \    {\bf x}\in {\mathbb R}^N,$$
and
the local objective functions $f_i, i\in V$, are twice continuously differentiable and satisfy Assumption  \ref{assump:f}
\cite{Zeidler1990, Emirov2022, Sun2014}.
For $1\le p\le \infty$, we denote the standard 
 $\ell^p$-norm of a graph signal ${\bf x}\in {\mathbb R}^N$ by $\|{\bf x}\|_p$.
 By Assumption \ref{assump:J}, we see that the gradient $\nabla F$ of the objective function $F$ has bi-Lipschitz property,
$$ c_1 \|{\bf x}-{\bf y}\|_2\le \|\nabla F({\bf x})-\nabla F({\bf y})\|_2\le  L_1 \|{\bf x}-{\bf y}\|_2\  \  {\rm for \ all} \ \  {\bf x}, {\bf y}\in {\mathbb R}^N.$$

\subsection{Fusion centers for distributed  and decentralized implementation}\label{fusioncenter.subsection}
In this paper, the proposed DAC algorithm  are implemented at fusion centers   $\Lambda\subseteq V$ with 
enough memory,  proper communication
bandwidth and high computing power.

Based on the distribution of fusion centers  $\Lambda$, we
divide the whole graph $V$ into  a family of nonoverlapping regions
$D_\lambda\subseteq V, \lambda\in \Lambda$,  around fusion centers that satisfies
\begin{equation}\label{governingvertices.def}
  \cup_{\lambda\in \Lambda} D_\lambda=V\ \  {\rm and} \  \ D_\lambda\cap D_{\lambda'}=\emptyset \ \ {\rm for \ distinct} \ \lambda, \lambda'\in \Lambda. 
  % \ \ {\rm and} \ \ \lambda\in D_\lambda\ {\rm for\ all}\ \lambda\in \Lambda.
\end{equation}
In practice, we may
require that the above {\em governing regions}  $D_\lambda, \lambda\in \Lambda$, %in the  above nonoverlapping network % to have similar sizes and are
are contained in some $R_0$-neighborhood of  fusion centers, i.e.,
$D_\lambda\subseteq B(\lambda, R_0), \ \lambda\in \Lambda$
for some $R_0\ge 0$.  A conventional selection of the above nonoverlapping network is
the Voronoi diagram of $\cG$ associated with  $\Lambda$, that satisfies
%\begin{equation*}
%\rho(\{i\} , U\backslash \{\lambda\})\le \rho(i, \lambda), \ i\in D_\lambda
%\end{equation*}
\begin{equation*}
  \big\{i\in V:\ \rho(i,\lambda)< \rho(\{i\} , \Lambda\backslash \{\lambda\}) \big\}\subseteq D_\lambda
  \subseteq \big\{i\in V:\ \rho(i,\lambda)\le  \rho(\{i\} , \Lambda\backslash \{\lambda\})\big\}
\end{equation*}
for all $\lambda\in \Lambda$,
 where we denote the distance between two vertex subsets  $V_1, V_2\subseteq V$ by $\rho(V_1, V_2)=\inf_{i\in V_1, j\in V_2} \rho(i,j)$.

Given  $R>0$,
we say that  $D_{\lambda, R}, \lambda \in \Lambda$ is an {\em extended $R$-neighborhood} of the nonoverlapping region $D_\lambda, \lambda\in \Lambda$, if
satisfies
\begin{equation}\label{extendedneighbor.def}
  D_\lambda\subseteq D_{\lambda, R} \ \ {\rm and} \ \ \rho(D_\lambda, V\backslash D_{\lambda, R})>  R \ \ {\rm for \ all}\  \lambda \in \Lambda.
  %\rho(i,j)>R \ {\rm if} \ i\in D_\lambda \ {\rm and}  \ j\in {\DlRc}.
\end{equation}
An example of extended $R$-neighborhoods
of the nonoverlapping network $D_\lambda, \lambda\in \Lambda$, is its $R$-neighborhoods
% define $R$-neighborhood of $D_{\lambda}, \lambda\in \Lambda$ by
$$B(D_\lambda, R)=\cup_{i\in D_\lambda} B(i, R)=\{j\in V: \ \rho(j,i)\le R\ {\rm for \ some} \ i\in D_\lambda\},\  \lambda\in \Lambda.$$
%be the set  of all $R$-neighboring vertices of $D_\lambda, \lambda\in \Lambda$.
%Clearly,  %$D_{\lambda, R}, \lambda\in \Lambda$   %is an overlapping covering of the whole network and satisfies
In\ the proposed  DAC algorithm described in the next section,  we divide
the  constrained convex optimization problem \eqref{convexoptimizationgeneral.def} into a family of
 interactive local  constrained minimization problems on extended $R$-neighbors
$D_{\lambda, R}, \lambda \in \Lambda$, and to solve the corresponding local constrained minimization problem at the fusion center $\lambda\in \Lambda$.

For $l\ge 0$, we denote the  $l$-neighbors  of $D_{\lambda, R}, \lambda\in \Lambda$, by
\begin{equation}
  \label{dlambdarl.def}
  D_{\lambda, R, l}=\{i\in V,\ \rho(\{i\}, D_{\lambda, R})\le l\}, \ \lambda\in \Lambda,
\end{equation}
and let
$\Lambda_{\lambda,  R,  4m}^{\rm out}=\{\lambda'\in \Lambda, D_{\lambda}\cap D_{\lambda', R, 4m}\ne \emptyset\}$
and
$\Lambda_{\lambda, R,  4m}^{\rm in}=\{\lambda'\in \Lambda, D_{\lambda'}\cap D_{\lambda, R, 4m}\ne \emptyset\}, \lambda \in \Lambda$ be
out-neighboring set and in-neighboring set of the fusion center $\lambda\in \Lambda$ respectively.
To implement our proposed DAC algorithm at the fusion centers, we require that each fusion center 
$\lambda\in \Lambda$
 be equipped with sufficient memory for data storage, adequate communication bandwidth for data exchange, and high computational capacity to solve local constrained minimization problems.

\begin{assumption}\label{fusioncenter.assump}
  (i)  {\rm (Memory)} \  Each fusion center $\lambda\in \Lambda$ can  store vertex sets $D_\lambda$, $D_{\lambda, R}$,  $D_{\lambda, R,  m}$, $D_{\lambda, R, 2m}$, $D_{\lambda, R, 4m}$, neighboring fusion  sets $\Lambda_{\lambda, R, 4m}^{\rm out}$ and $\Lambda_{\lambda, R, 4m}^{\rm in}$, and the vectors  $\chi_{D_{\lambda, R, 4m} } {\bf x}^{(n)}$  and %$\chi_{D_{\lambda, R}}
  ${\bf w}_\lambda^{(n)}$ at each iteration, and reserve enough memory used for storing local objective functions $f_i, i\in D_{\lambda, R, m}$ and solving the local minimization problem \eqref{DACwithoutinequality.deflocala} at each iteration.

  (ii) {\rm (Computing power)} \ Each fusion center $\lambda\in \Lambda$ has computing facility to solve the local minimization problem  \eqref{DACwithoutinequality.deflocala}.

  (iii) {\rm (Communication bandwidth)} \ Each fusion center $\lambda\in \Lambda$ can send data $x_i^{(n)}, i\in D_\lambda$ to fusion centers $\lambda' \in  \Lambda_{\lambda, R, 4m}^{\rm out}$ and receive data $x_i^{(n)}, i\in D_{\lambda'}$ from fusion centers $\lambda' \in  \Lambda_{\lambda, R, 4m}^{\rm in}$ at each iteration.
\end{assumption}

\subsection{Constraints}\label{constraints.section}
In this paper, we require that the constraint functions $g_l, l\in U\subseteq V$, in  the
 constrained optimization \eqref{convexoptimizationgeneral.def}, are smooth,  convex and 
 centered around the vertices in $U$.

\begin{assumption}\label{assum:g_banded} For every $l\in U$, the local constraint function $g_{l}({\bf x})$   is convex and twice continuously differentiable, and  depends only on $x_j, j\in B(l,m)$, where ${\bf x}=[x_j]_{j\in V}$ with $x_j\in {\mathbb R}, j\in V$, and
  the integer $m$ is the same as neighboring radius of the local objective functions $f_i, i\in V$.
\end{assumption}

Similar to the objective functions, we can show that the neighboring requirement
for the  constraint function $g_l, l\in U$, in Assumption  \ref{assum:g_banded} can be described as
\begin{equation}\label{assumption2.characterization+}
  \omega\Big( \Big [\frac{\partial g_l({\bf x})}{\partial x_j}\Big ]_{l\in U, j\in V}\Big)\le m
  \ \ {\rm for \ all} \  {\bf x}=[x_j]_{j\in V}\in {\mathbb R}^N. 
\end{equation}
%where ${\bf x}=(x_j)_{j\in V}\in {\mathbb R}^S$.
 
 \smallskip
 
In this paper, we require that the constrained matrix ${\bf A}=[ a(k, j)]_{k\in W, j\in V}$ in the
constrained convex optimization problem \eqref{convexoptimizationgeneral.def} is of locally full-rank,
where 
 $a(k,j)\in {\mathbb R}$ for $k\in W, j\in V$:
 
\begin{assumption}\label{assump:A}
(i) The constraint at each vertex $k\in W\subset V$ is local, i.e.,
  $ a(k, j)=0$ for  all $k\in W$ and $j\in V$  satisfying $\rho(k,j)>2m$,
  or equivalently,
  \begin{equation} \label{assump:A2}
    \omega({\bf A}) \le 2m,
  \end{equation}
  where $m\ge 1$ is
  the neighboring radius of the local objective functions $f_i, i\in V$ in Assumption \ref{assump:f}.
  
 (ii) The constrained matrix ${\bf A}$ has the  local stability:
  \begin{equation}\label{assump:A3}
    {\bf A}_{\lambda, R} {\bf A}_{\lambda, R}^T\succeq { c}_2^2\ {\bf I}_{W_{\lambda, R}}, \ \lambda\in \Lambda,
  \end{equation}
hold for some positive constant $c_2$, where
${\bf A}_{\lambda, R}= [a(k,j)]_{k\in W_{\lambda, R}, j\in D_{\lambda, R}}$ is local
constrained matrix associated with the extended neighbor $D_{\lambda, R}$ and
\begin{align}\label{eq:Gamma_lambda}
  W_{\lambda, R} = \{k\in  W: \  a(k,j) \neq 0 \text{ for some } j\in \DlR\}.
\end{align}
\end{assumption}

\section{Divide-and-conquer algorithm and exponential convergence}\label{dac.section}

In this section, we propose an iterative divide-and-conquer algorithm
to solve the constrained convex optimization problem \eqref{convexoptimizationgeneral.def}
and establish its exponential convergence. For this purpose, we first examine the scenario without inequality constraints in Section \ref{sec:equality}, and
subsequently extend the framework to accommodate inequality constraints via the log-barrier method in Section  \ref{sec:inequality}.

\subsection{Constrained optimization problems without inequality constraints}\label{sec:equality}

In this subsection, we propose an iterative DAC algorithm to solve the  convex optimization problem \eqref{convexoptimizationgeneral.def}
%on the graph ${\mathcal G}$ 
without the inequality constraints \eqref{convexoptimizationgeneral.def3}:
\begin{equation} \label{convexoptimizationlinearconstraint.def}
  {\bf x}^*=\arg\min_{{\bf x}\in  {\mathbb R}^N} F({\bf x})\ \ {\rm subject \ to }\ \
  {\bf A} {\bf x} = {\bf b}.
\end{equation}

To this end, we introduce a projection map to decompose the global optimization problem \eqref{convexoptimizationlinearconstraint.def} into a family of local subproblems, together with its adjoint map, which aggregates the local solutions into a coherent global result. For a subset $V_1\subseteq V$, we use $\chi_{_{V_1}}: {\mathbb R}^N\rightarrow {\mathbb R}^{N_1}$ to denote the selection map $\chi_{_{V_1}} {\bf x} = [x_i]_{i\in V_1}$ for ${\bf x}=[x_i]_{i\in V} \in {\mathbb R}^{N}$,
and define its adjoint map
$\chi^*_{V_1}: {\mathbb R}^{N_1}\rightarrow {\mathbb R}^{N}$ by $\chi^*_{V_1} {\bf u} = \sum_{i\in V_1} u_i e_i$ for ${\bf u}=[u_i]_{i\in V_1}\in {\mathbb R}^{N_1}$, where $e_i$ is the $i$-th unit vector in ${\mathbb R}^N$ and $N_1$ is the cardinality of the set $V_1$. % is used to aggregate local solutions. 
For $W\subseteq V$, 
  the projection operator \(\mI_{W} = \chi_{W}^{*}\chi_{W}\) makes the components of a vector \({\bf x}\) zero if the corresponding vertex is not in the set \(W\).

We next describe the proposed DAC algorithm. Let $\Lambda\subseteq V$ be the set of fusion centers, $\{D_\lambda, \lambda\in \Lambda\}$ be nonoverlapping regions
satisfying \eqref{governingvertices.def}, $D_{\lambda, R}$ and $W_{\lambda, R}, \lambda\in \Lambda$ be given in \eqref{extendedneighbor.def} and \eqref{eq:Gamma_lambda} respectively; see 
Section \ref{fusioncenter.subsection}. Starting from an arbitrary initial point ${\bf x}^{(0)}\in {\mathbb R}^N$, at each iteration $n\ge 0$ we break down the optimization problem \eqref{convexoptimizationlinearconstraint.def}
into a family of local  constrained optimization problems on overlapping extended $R$-neighbors $D_{\lambda, R}$ of the fusion centers $\lambda\in \Lambda$, and solve them in a distributed and decentralized manner for every $\lambda\in \Lambda$:
\begin{align}\label{DACwithoutinequality.defb}
 \begin{dcases} 
        {\bf w}_{\lambda}^{(n)}  =   \mathop{\arg\min}\limits_{{\bf u} }  F(\chi^*_{_{D_{\lambda, R}}} {\bf u} + \mI_{_{V\backslash D_{\lambda, R}}} {\bf x}^{(n)}) \\
    \quad  {\rm s. t. }  \  \chi_{_{W_{\lambda, R}}} {\bf A} (\chi^*_{_{D_{\lambda, R}}} {\bf u} +\mI_{_{V\backslash D_{\lambda, R}}}  {\bf x}^{(n)}) = \chi_{_{W_{\lambda, R}}}{\bf b}.
    \end{dcases}
\end{align}
We then use local solutions ${\bf w}_{\lambda}^{(n)}, \lambda\in \Lambda$, to update
  \begin{align}\label{DACwithoutinequality.defd}
    {\bf x}^{(n+1)} = \sum_{\lambda\in \Lambda} \mI_{_{D_{\lambda}}} \chi^*_{\DlR} {\bf w}_\lambda^{(n)}.
  \end{align}

% Write  ${\bf A}=(a(k, j))_{k\in W, j\in V},   {\bf b}=(b_k)_{k\in W}$,
Write  ${\bf x}^{(n)}=[x_i^{(n)}]_{i\in V}$ and ${\bf w}_{\lambda}^{(n)}=[w_{\lambda, i}^{(n)}]_{i\in D_{\lambda, R}}$ for $ n\ge 0$,
 and define $D_{\lambda, R, l}, l\ge 0$  as in \eqref{dlambdarl.def}.
Following the arguments in \cite{Emirov2022} and applying Assumptions \ref{assum:g_banded} and \ref{assump:A}, we can reformulate the local constrained optimization problem
\eqref{DACwithoutinequality.defb} as
\begin{align}\label{DACwithoutinequality.deflocala}
  \begin{dcases}
   \big(w_{\lambda, i}^{(n)}\big)_{i\in D_{\lambda, R}}   =
  %\begin{dcases}
    \arg\min_{\bf u}  \sum_{j\in  D_{\lambda, R, m}}
  f_j\big(\chi^*_{D_{\lambda, R}}{\bf u}+ \mI_{_{D_{\lambda, R, 2m}\backslash D_{\lambda, R}}}{\bf x}^{(n)}\big)\\
 \quad  {\rm s. t. }
  \sum_{j\in D_{\lambda, R}} a(k, j) u_j=b_k-\sum_{j\in D_{\lambda, R, 4m} \backslash D_{\lambda, R}} a(k,j) x_j^{(n)},\ k\in W_{\lambda, R},
  \end{dcases}
\end{align}
%where $D_{\lambda, R, 4m}$ is defined as in \eqref{dlambdarl.def}.
%$$b_k-\sum_{j\in D_{\lambda, R, 4m} \backslash D_{\lambda, R}} a(k,j) x_j^{(n)},\ k\in W_{\lambda, R}$$
and we can rewrite the updating procedure \eqref{DACwithoutinequality.defd} as
\begin{equation}
  \label{DACwithoutinequality.deflocalb}
  x^{(n+1)}_i= w_{\lambda, i}^{(n)} \ \ {\rm if} \ i\in D_{\lambda},
\end{equation}
where the above update is well defined by \eqref{governingvertices.def} and \eqref{extendedneighbor.def}.
We call the above procedure the {\em divide-and-conquer} (DAC) algorithm.

For ${\bf u}=[u_i]_{i\in D_{\lambda, R}}$,  we observe that
$\chi^*_{D_{\lambda, R}}{\bf u}+ \mI_{_{D_{\lambda, R, 2m}\backslash D_{\lambda, R}}}{\bf x}^{(n)}$
has its $i$-th components taking the value $ u_i$ for $i\in D_{\lambda, R}$, ${ x}_i^{(n)}$ for $i\in
  D_{\lambda, R, 2m}\backslash D_{\lambda, R}$ and $0$ for all other vertices $i\in V\backslash D_{\lambda, R, 2m}$.
This observation, together with
\eqref{DACwithoutinequality.deflocala} and \eqref{DACwithoutinequality.deflocalb} implies that
the proposed DAC algorithm \eqref{DACwithoutinequality.defb} can be implemented
in fusion centers $\lambda\in \Lambda$ equipped with enough processing power and resources. We present the implementation of the proposed DAC algorithm in every fusion center in Algorithm \ref{Main.algorithm}.

\begin{algorithm}[t] %[h!] %[h!]
  \caption{Implementation of the DAC algorithm \eqref{DACwithoutinequality.defb} and \eqref{DACwithoutinequality.defd} at a fusion center $\lambda\in \Lambda$.}
  \label{Main.algorithm}
  \begin{algorithmic}  %[1]
  \STATE{\bf Initialization}:  Maximum iteration count $T$; vertex sets $D_\lambda$, $D_{\lambda, R}$, $D_{\lambda, R, 2m}$ and
    $D_{\lambda, R, 4m}$, neighboring fusion  sets
    $\Lambda_{\lambda, R, 4m}^{\rm out}$ and
    $\Lambda_{\lambda, R, 4m}^{\rm in}$,
    local objective functions $f_i, i\in D_{\lambda, R, m}$,
    and initial guess $\chi_{_{D_{\lambda, R, 4m}}}{ \bf x}^0$, i.e., its components $x^{(0)}_i, i\in D_{\lambda, R, 4m}$.
    \STATE{\bf Iteration}:

    {\bf for} \ $n=0, 1, \ldots, T$
    \begin{itemize}
  \item  Solve the local constrained minimization problem
            \eqref{DACwithoutinequality.deflocala} to obtain
            ${\bf w}_{\lambda}^{(n)}=[w_{\lambda, i}^{(n)}]_{i\in D_{\lambda, R}}$.
            %\begin{align*}
            %{\bf w}_{\lambda}^n  = \arg\min_{{\bf u} \in \bR^{\# D_{\lambda, R}}}
            %\sum_{i\in D_{\lambda, R, m}} f_{i}(\chi^*_{D_{\lambda, R}}{\bf u}+\mI_{D_{\lambda, R, 2m}\backslash D_{\lambda, R}}{\bf x}^{n})
            %\end{align*}

      \item Update ${x}^{(n+1)}_i$ by  $w_{\lambda, i}^{(n)}$ for $i\in D_\lambda$; see \eqref{DACwithoutinequality.deflocalb}.

      \item Send data ${x}^{(n+1)}_i, i\in D_\lambda$, to all out-neighboring fusion centers $\lambda'\in \Lambda_{\lambda,  R,  4m}^{\rm out}$.
            % with $D_{\lambda', R, 4m}\cap D_{\lambda}\ne \emptyset$.
            %\in \Lambda_{\lambda, R, 2m}^{\rm out}$;

      \item Receive data ${x}^{(n+1)}_j, j\in D_{\lambda'}$ from in-neighboring fusion centers $\lambda'\in \Lambda_{\lambda,  R,  4m}^{\rm in}$
            % with  $D_{\lambda, R, 4m}\cap D_{\lambda'}\ne \emptyset$
            and obtain $\chi_{_{D_{\lambda, R, 4m}}} {\bf x}^{(n+1)}$.

            %\item[{}] Evaluate $\chi_{D_{\lambda, R, 2m} \backslash D_{\lambda, R} } {\bf x}^{n+1}  = %\chi_{D_{\lambda}} \vw_{\lambda}^n+
            %\chi_{D_{\lambda, R, 2m}\backslash D_{\lambda, R} } \sum_{\lambda'\in \Lambda_{\lambda, R,  2m}^{\rm in}}
            %\mI_{D_{\lambda'}}  \chi^*_{D_{\lambda',R}}{\bf w}_{\lambda'}^n$.

    \end{itemize}
    {\bf end}
    \STATE{\bf Output}:  $\chi_{_{D_{\lambda, R, 4m}}} {\bf x}^{(T+1)}$,  local approximation to the minimizer ${\bf x}^*$ of the
    constrained optimization problem \eqref{convexoptimizationlinearconstraint.def}.

    % $\chi_{D_{\lambda, R, 2m}} \vx^{M+1}$.
  \end{algorithmic}\vspace{-.03in}
\end{algorithm}

 Denote the operator norm of a matrix ${\bf A}$ by   \begin{equation}\label{assump:A1}
    \|{\bf A}\|:=\sup_{\|{\bf x}\|_2=1} \|{\bf A}{\bf x}\|_2<\infty.
  \end{equation}
  For the proposed DAC algorithm described in \eqref{DACwithoutinequality.defb} and \eqref{DACwithoutinequality.defd}, we show that it converges exponentially to the solution of the original constrained optimization problem \eqref{convexoptimizationlinearconstraint.def}
when the parameter $R$ is appropriately chosen; see Section \ref{maintheorem.thm.pfsection} for  the detailed proof.

\begin{theorem}\label{maintheorem1.thm}
  Suppose that Assumptions \ref{PolynomialGrowth.assumption}, \ref{assump:f}, \ref{assump:J}  %, \ref{fusioncenter.assump} 
  and \ref{assump:A} hold.
  Consider the constrained convex optimization problem \eqref{convexoptimizationlinearconstraint.def}
  and denote its unique minimizer by ${\bf x}^*$.
  Set
       \begin{equation}\label{kappa.def}
    \kappa=  \Big(\frac{c_1+\|{\bf A}\|}{c_1}\Big)^2 (L_1+\|{\bf A}\|) \max\Big( \frac{1}{c_1}, \frac{L_1}{c_2^2}\Big),
  \end{equation}
  and
  \begin{equation}\label{maintheorem1.thm.eq1}
    \delta_R= D_1({\mathcal G}) d!
    \Big(\frac{1}{4m}\ln \Big(\frac{\kappa^2-1}{\kappa^2+1}\Big)\Big)^{-d} (R+2)^{d} \Big(\frac{\kappa^2-1}{\kappa^2+1}\Big)^{R/(4m)}, \ R\ge 1,
  \end{equation}
  where
  $d:=d(\cG)$ and $D_1(\cG)$ are Beurling dimension and density of the graph ${\mathcal G}$ respectively,
   $m, c_1,  c_2$ and $L_1$  are constants in Assumptions \ref{assump:f}, \ref{assump:J} and \ref{assump:A}. 
    Let ${\bf x}^{(n)}, n\ge 0$, be the sequence generated in the iterative algorithm \eqref{DACwithoutinequality.defb}. If the radius parameter $R$ is chosen such that
  \begin{equation}\label{deltaR.eq} \delta_R<1,
  \end{equation}
  %$ in \eqref{maintheorem1.thm.eq1} is less than 1,
  then ${\bf x}^{(n)}, n\ge 0$, converges to ${\bf x}^*$ exponentially in $\ell^p, 1\le p\le \infty$, with convergence rate $\delta_R$:
  \begin{equation}\label{maintheorem1.thm.eq4}
    \|{\bf x}^{(n)} - {\bf x}^*\|_p \le (\delta_R)^n \|{\bf x}^{(0)} - {\bf x}^*\|_p, \quad n\ge 0.
  \end{equation}
\end{theorem}

\medskip

Next, we consider the scenario in which the numerical solutions to the local optimization problems \eqref{DACwithoutinequality.defb} are inexact, with accuracy  up to $\epsilon_n$ at the $n$-th iteratiion, 
 as typically occurs when iterative solvers are used.
%In particular, we assume the corresponding Karush–Kuhn–Tucker  conditions are satisfied up to some accuracy $\epsilon_n$.
In particular, for each $\lambda\in \Lambda$ and $n\in \bN$, we assume the approximate solutions $\widetilde{{\bf w}}_\lambda^{(n)}$ of 
the local linearly constrained optimization \eqref{DACwithoutinequality.defb} satisfies
up to some accuracy $\epsilon_n$
\begin{align*}
  \| \vtheta_{\lambda}^{(n)}\|_\infty \le \epsilon_n \quad \mbox{and}\quad \|\veta_{\lambda}^{(n)}\|_\infty\le \epsilon_n,
\end{align*}
where 
\begin{subequations}\label{inexactmaintheorem1.pf.eq2}
  \begin{equation}
    \label{inexactmaintheorem1.pf.eq2a}
    {\pmb \theta}_\lambda^{(n)}:=\chi_{_{\DlR}} \nabla F\big(\chi^*_{_{\DlR}} {\tilde {\bf w}}_\lambda^{(n)} + {\bf I}_{{\DlRc}} {\bf x}^{(n)}\big)+\chi_{_{\DlR}}{\bf A}^T\chi^*_{_{W_{\lambda, R}}} \tilde {\bf v}^{(n)}_{\lambda},  \end{equation}
    \begin{equation}
    \label{inexactmaintheorem1.pf.eq2b}
    {\pmb \eta}_\lambda^{(n)}:=\chi_{_{W_{\lambda, R}}} {\bf A} \big(\chi^*_{_{\DlR}} {\tilde {\bf w}}_\lambda^{(n)} + {\bf I}_{\DlRc} {\bf x}^{(n)}\big)-\chi_{_{W_{\lambda, R}}} {\bf b},
  \end{equation}
\end{subequations}
and ${\tilde {\bf v}}_\lambda^{(n)}$ is an approximate multiplier in the Karush–Kuhn–Tucker conditions for the local linearly constrained optimization \eqref{DACwithoutinequality.defb}.
We then combine the local solutions to get the global update
\begin{align}\label{eq:tldxn+1}
  \widetilde{{\bf x}}^{(n+1)} = \sum_{\lambda\in \Lambda} \mI_{\Dl} \chi^*_{\DlR} \widetilde{{\bf w}}_\lambda^{(n)}.
\end{align}
In the following theorem, we estimate the approximation error of the sequence $\{\tilde{{\bf x}}^{(n)}\}$ to the true solution ${\bf x}^*$; see Section 
\ref{inaxactmaintheorem.thm.pfsection} 
for the detailed proof. 

\begin{theorem}\label{inexactmaintheorem1.thm}  Suppose that Assumptions \ref{PolynomialGrowth.assumption}, \ref{assump:f}, \ref{assump:J}  %, \ref{fusioncenter.assump} 
  and \ref{assump:A} hold.
 Let ${\bf x}^*$ be the unique minimizer of the global optimization problem \eqref{convexoptimizationgeneral.def} and $\{\tilde{{\bf x}}^{(n)}\}$ be the  inexact solution sequence generated in the iterative algorithm \eqref{eq:tldxn+1}. If the parameter $R$ is chosen such that $\delta_R$ in \eqref{maintheorem1.thm.eq1} is strictly less than 1, then
  \begin{equation}\label{inexactmaintheorem1.thm.eq1}
    \|\widetilde{{\bf x}}^{(n)} - {\bf x}^*\|_\infty \le (\delta_R)^n \|\tilde{{\bf x}}^{(0)} - {\bf x}^*\|_\infty + \frac{2}{\sqrt{L}} \sum_{m=0}^{n-1}(\delta_R)^{n-m}\epsilon_m , \quad n\ge 1.
  \end{equation}
\end{theorem}

\subsection{Constrained optimization problems with inequality constraints}\label{sec:inequality}

 In this section, we employ the log-barrier method \cite{YuriiNesterov1994} to solve  the general constrained optimization problem \eqref{convexoptimizationgeneral.def} with inequality constraints. 
In particular, we will solve the following equality constrained problem instead:
\begin{align}\label{eq:globalopt_logbarrier}
  {\rm argmin}_{{\bf x}\in  {\mathbb R}^S} F_{t}({\bf x}) := F({\bf x}) + G_t({\bf x}) \ \  \st\ \  {\bf A} {\bf x} = {\bf b},
\end{align}
where $G_t({\bf x}):=t^{-1} \sum_{l\in U} - \log(-g_l({\bf x}))$.
It is well known \cite{YuriiNesterov1994} that the solution ${\bf x}^*_t$ of the log-barrier problem \eqref{eq:globalopt_logbarrier} approximates the solution ${\bf x}^*$ of the inequality constrained problem \eqref{convexoptimizationgeneral.def} when $t>0$ is sufficient large,
\begin{align*}
  F({\bf x}^*_t) - F({\bf x}^*) \le  N t^{-1}, %\frac{\# V}{t}.
\end{align*}
where $N$ is the order of the underlying graph ${\mathcal G}$ of  the general constrained optimization problem \eqref{convexoptimizationgeneral.def}.
Then we employ the DAC algorithm proposed in Section \ref{sec:equality} to solve the log-barrier problem \eqref{eq:globalopt_logbarrier}. Specifically, we apply \eqref{DACwithoutinequality.defb} on the new objective function $F_{t}({\bf x})$ in \eqref{eq:globalopt_logbarrier} and  solve the following local optimization problem for each $\lambda\in \Lambda$,
\begin{align*}
  {\bf w}_{t,\lambda}^{(n)} =
  \begin{dcases}
    \arg\min_{{\bf u}} % \in \bR^{\# \DlR}} 
    F_t (\chi^*_{\DlR} {\bf u} + \mI_{V\backslash\DlR} {\bf x}_t^{(n)}) \\
     \text{s.t. }   \chi_{\Gl} {\bf A} (\chi^*_{\DlR} {\bf u} + \mI_{V\backslash\DlR} {\bf x}_t^{(n)}) = \chi_{\Gl}{\bf b}
  \end{dcases}
\end{align*}
and combine those local solutions to obtain the next updated value
\begin{align}\label{eq:xt_n+1}
  {\bf x}_t^{(n+1)} = \sum_{\lambda\in \Lambda} \mI_{\Dl} \chi^*_{\DlR} {\bf w}_{t,\lambda}^{(n)}.
\end{align}
For the above DAC algorithm, we have a similar approximation error estimate as in  Theorem  \ref{maintheorem1.thm}, see Section \ref{thm:logbarrier.pfsection} for detailed proof. 

\begin{theorem}\label{thm:logbarrier}
  Let  $\cG=(V, E)$ be a simple graph of order $N$ that has the  polynomial growth property .
  Consider the log-barrier problem \eqref{eq:globalopt_logbarrier} and denote its unique minimizer  by ${\bf x}_t^*$.
  Suppose  Assumptions  \ref{assump:f},  \ref{assump:J}, \ref{fusioncenter.assump}, \ref{assum:g_banded} and \ref{assump:A} hold.
  Set
    \begin{equation}\label{thm:logbarrier.eq2}
    \kappa_t=  \Big(\frac{c_1+\|{\bf A}\|}{c_1}\Big)^2 (L_1 + M_t +\|{\bf A}\|) \max\Big( \frac{1}{c_1}, \frac{L_1 + M_t}{c_2^2}\Big)
  \end{equation}
and \begin{equation}\label{thm:logbarrier.eq1}
    \delta_{R,t}= D_1({\mathcal G}) d!
    \Big(\frac{1}{4m}\ln \Big(\frac{\kappa_t^2-1}{\kappa_t^2+1}\Big)\Big)^{-d} (R+2)^{d({\mathcal G})} \Big(\frac{\kappa_t^2-1}{\kappa_t^2+1}\Big)^{R/(4m)},
  \end{equation}
  where $d:=d(\cG)$ and $D_1(\cG)$ are Beurling dimension and density of the graph ${\mathcal G}$ respectively,
  $c_1, c_2, L_1$ are constants in Assumptions \ref{assump:J} and \ref{assump:A}, $M_t$ is the upper bound of $\nabla^2 G_t$ on a sublevel set $\{{\bf x}\in \bR^{N}: F({\bf x}) + G_t({\bf x})\le F({\bf x}_t^{(0)}) + G_t({\bf x}_t^{(0)}) \}$ for an initial feasible point ${\bf x}_t^{(0)}$.
  Let $\{{\bf x}_t^{(n)}\}$ be the sequence generated in the iterative algorithm \eqref{eq:xt_n+1}. If the parameter $R$ is chosen such that 
  $$\delta_{R, t}<1, $$ then $\{{\bf x}_t^{(n)}\}$ converges to ${\bf x}_t^*$ exponentially in $\ell^p, 1\le p\le \infty$ with convergence rate $\delta_R$:
  \begin{equation*}
    \|{\bf x}_t^{(n)} - {\bf x}_t^*\|_p \le (\delta_{R, t})^n \|{\bf x}_t^{(0)} - {\bf x}_t^*\|_p, \quad n\ge 0.
  \end{equation*}

\end{theorem}

\section{Numerical simulations}\label{numericalexperiments.section}

In this section, we evaluate the performance of the proposed DAC algorithm on several constrained optimization problems, that have been widely used in various applications, such as portfolio optimization  in finance \cite{Markowitz1952}, multinomial logistic regression with identifiability constraints to avoid redundancy in categorical models \cite{Agresti2002}, the direct current (DC) approximation of power flow optimization problem \cite{Wood2013}, neural network weight sharing via constraints \cite{Caruana1997}, and the entropy maximization problem \cite{Jaynes1957}. We will test our proposed algorithm on three popular types of loss functions: a \(L_{2}\) distance function, a quadratic function, and an entropy function. For the distance function and the quadratic function, we will include equality constraints in the model. For the cross entropy function, we will include both equality constraints and inequality constraints.

For all the simulations in this section, we consider a random geometric graph $\cG=(V, E)$ with $N$ vertices uniformly distributed in the unit square $[0,1]^2$, where two vertices are connected by an edge if their Euclidean distance is less than a given threshold $\tau$. Here we set $\tau=\sqrt{3N^{-1}\log N}$ to ensure that the random geometric graph is connected with high probability \cite{Penrose2003, Emirov2022}. 

We will use the algorithm introduced in \cite{Emirov2022} to find the fusion centers $\Lambda$.  In particular,  we randomly select a vertex $i\in V$ and add it to the fusion center set $\Lambda$. We choose \(R=1\) and then remove all vertices in its $2R$-neighborhood $B(i, 2R)$ from $V$. We repeat this process until all vertices in $V$ are removed. The selected fusion centers $\Lambda$ satisfy  Assumption \ref{governingvertices.def}.

For the vertices set \(W\) with constraints, we randomly select 10\% of the vertices in $V$ and combine them with the fusion centers \(\Lambda\) to get the set \(W\).

\subsection{\(L_{2}\) distance loss}
In our first simulation, we consider a simple \(L_{2}\) distance function as the loss function, constrained by a set of linear equations. In particular, we consider the following optimization problem,
$$
\min_{{\bf x}} \frac12\|{\bf x}-{\bf z}\|^2 \ \text{ subject to } \ {\bf Ax }={\bf b}.
$$
We point out that the above optimization problem can be interpreted as an orthogonal projection problem onto the affine set \cite{Plesnik2007}. With the full rank assumption on the matrix ${\bf A}$, the solution of the above orthogonal projection problem is given by
$ [{\bf I} - {\bf A}^T ({\bf A}{\bf A}^T)^{-1}{\bf A}]{\bf z}+ {\bf A}^T ({\bf A}{\bf A}^T)^{-1}{\bf b}$. We will use this closed-form solution to evaluate the performance of our proposed DAC algorithm.

In our simulation,  we first compute the graph Laplacian matrix ${\bf L}_{\mathcal{G}}$ of the random geometric graph $\mathcal{G}=(V, E)$, and then we take ${\bf A}=\chi_{W}{\bf L}_{\mathcal{G}}$ and ${\bf b}={\bf 0}$, where $\chi_{W}$ is a truncation operator on $W$. The constraint $\chi_{W}{\bf L}_{\mathcal{G}}{\bf x} = {\bf 0}$ in our optimization problem can be interpreted as an averaging equality constraint where each value $x_i$, $i\in W$ should be the average of the neighboring variables. That is, our optimization problem has the following form
$$\min_{{\bf x}} \frac12\|{\bf x}-{\bf z}\|^2 \quad \text{ subject to } \chi_{W}{\bf L}_{\mathcal{G}}{\bf x} = {\bf 0}
$$
where the local objective function becomes $f_i({\bf x} ) = \sum_{j\in {\mathcal N}_i} \frac{1}{2 d_j}
(x_j-z_j)^2 $ for $i\in V$ and $m=1$, where ${\mathcal N}_i$ is the set of all vertices $j$ with $(i,j)$ being an edge and $d_i$
is the cardinality of ${\mathcal N}_i$ (also known as the degree of the vertex $i$). The vector ${\bf z}$ is randomly selected with each component uniformly distributed in $[0,1]$.

We apply the proposed DAC algorithm to solve the above orthogonal projection problem on the random geometric graph of order \(N=1024\)  and \(N=2048\). We compute the approximation error \(\lVert {\bf x}^{(n)} - {\bf x}^* \rVert_2\) at each iteration, where \({\bf x}^{(n)}\) is the solution obtained from the proposed DAC algorithm at the \(n\)-th iteration and \({\bf x}^*\) is the closed-form solution of the orthogonal projection problem. We run 100 trials and take the average of the approximation errors at each iteration to verify the robustness of the proposed DAC algorithm. We display the results in Figure \ref{fig:orthogonal_proj_plot}. This demonstrates that the proposed DAC algorithm converges to the optimal solution with an exponential rate.

\begin{figure}[t] %[!h]
  \begin{subfigure}[t]{0.48\textwidth}
    \includegraphics[width = \textwidth]{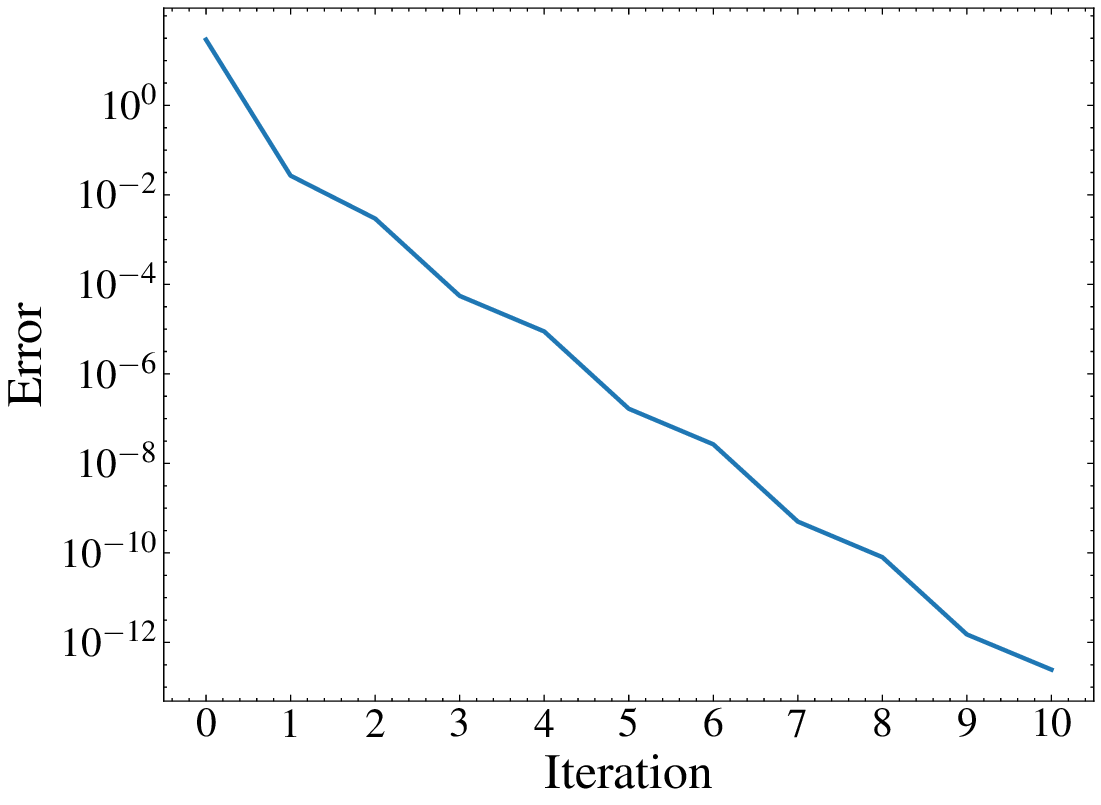}
    \label{fig:orthogonal_proj_plot_1024}
    \caption{\(N=1024\)}
  \end{subfigure}
  \begin{subfigure}[t]{0.48\textwidth}
    \includegraphics[width = \textwidth]{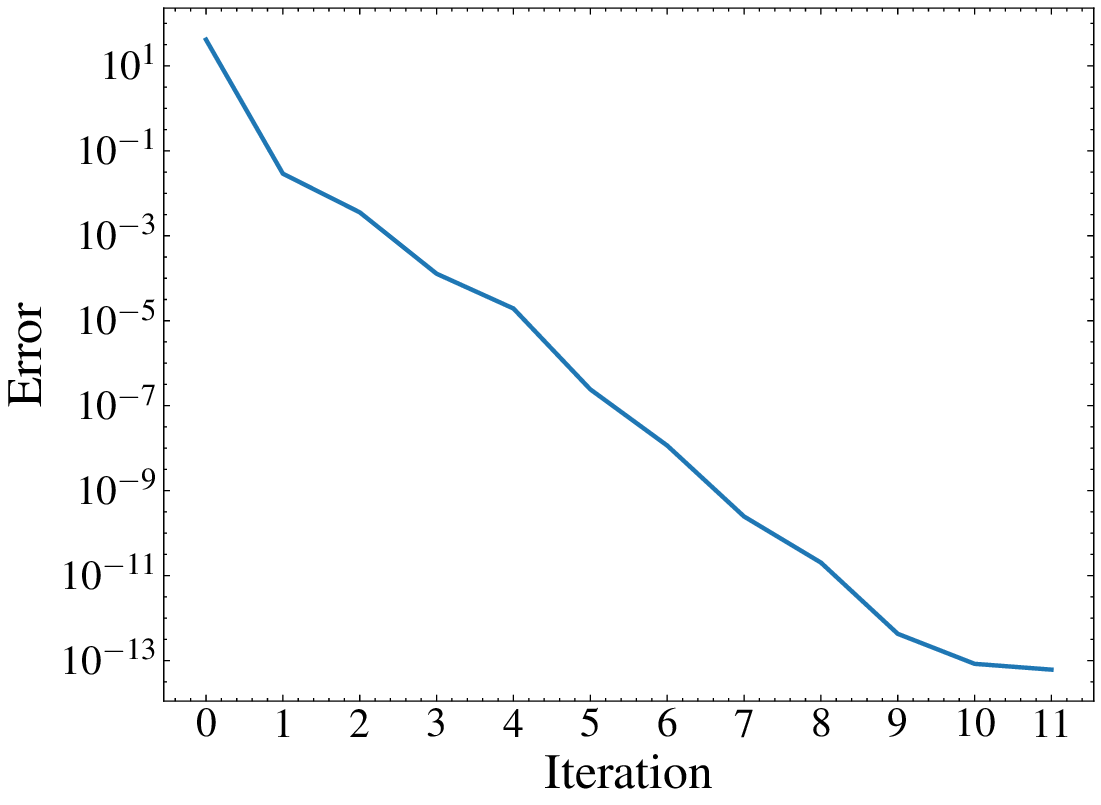}
    \label{fig:orthogonal_proj_plot_2048}
    \caption{\(N=2048\)}
  \end{subfigure}
  \caption{Approximation error \(\left\lVert {\bf x}^{(n)} - {\bf x}^* \right\rVert_2\) for the orthogonal projection problem.}
  \label{fig:orthogonal_proj_plot}
\end{figure}

\subsection{Quadratic loss}
In the second simulation, we consider a quadratic loss function as the objective function, constrained by a set of linear equations. In particular, we consider the following optimization problem,
$$
\min_{{\bf x}} \frac12 {\bf x}^T {\bf Q} {\bf x} + {\bf c}^T {\bf x} \quad \text{ subject to }\  {\bf Ax }={\bf b},
$$
where ${\bf Q}$ is a positive definite matrix and ${\bf c}$ is a given vector. Many application problems can be formulated as the above constrained quadratic optimization problem, including the portfolio optimization problem  in finance \cite{Markowitz1952}, the direct current approximation of power flow optimization problem \cite{Wood2013}, and traffic assignment problem \cite{sheffi1985urban}.

In our simulation, we set ${\bf Q}=4{\bf I}+{\bf L}_{\mathcal{G}}$, ${\bf c}$ as a random vector with each component uniformly distributed in $[0,1]$, ${\bf A}=\chi_{W}({\bf L}_{\mathcal{G}}^{2} + 2{\bf I})$ and ${\bf b}={\bf 0}$. 
Similar to the first simulation, we apply the proposed DAC algorithm to solve the above constrained quadratic optimization problem on the random geometric graph of order \(N=1024\)  and \(N=2048\). We compute the approximation error \(\lVert {\bf x}^{(n)} - {\bf x}^* \rVert_2\) at each iteration, where \({\bf x}^{(n)}\) is the solution obtained from the proposed DAC algorithm at the \(n\)-th iteration and \({\bf x}^*\) is the closed form solution of the constrained quadratic optimization problem obtained from the KKT conditions. Specifically, the closed form solution is given by
$$
{\bf x}^* = {\bf Q}^{-1} \big({\bf A}^T ({\bf A}{\bf Q}^{-1}{\bf A}^T)^{-1}({\bf b}+{\bf A}{\bf Q}^{-1}{\bf c}) - {\bf c}\big).
$$
Figure \ref{fig:quadratic_func_plot} presents the average approximation error over 100 trials, demonstrating that the proposed DAC algorithm converges  to the optimal solution with an exponential rate.

\begin{figure}[t] %[!h]
  \begin{subfigure}[t]{0.48\textwidth}
    \includegraphics[width = \textwidth]{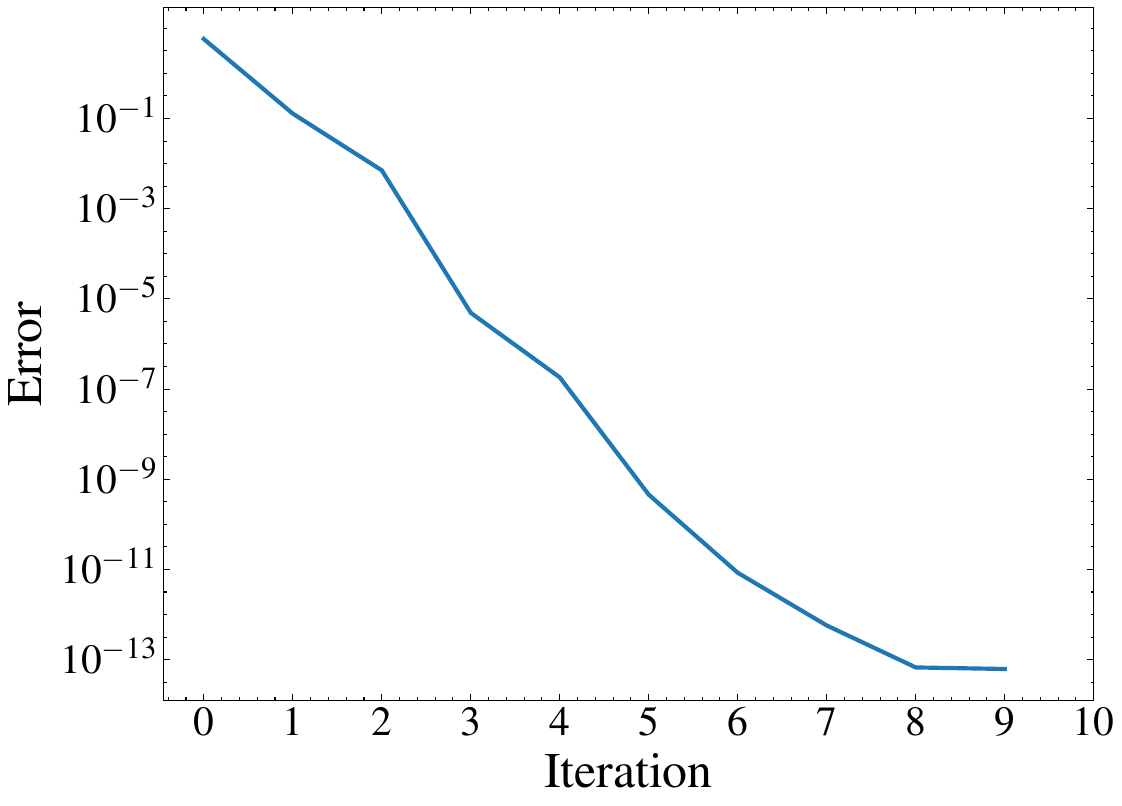}
    \label{fig:quadratic_func_plot_1024}
    \caption{\(N=1024\)}
  \end{subfigure}
  \begin{subfigure}[t]{0.48\textwidth}
    \includegraphics[width = \textwidth]{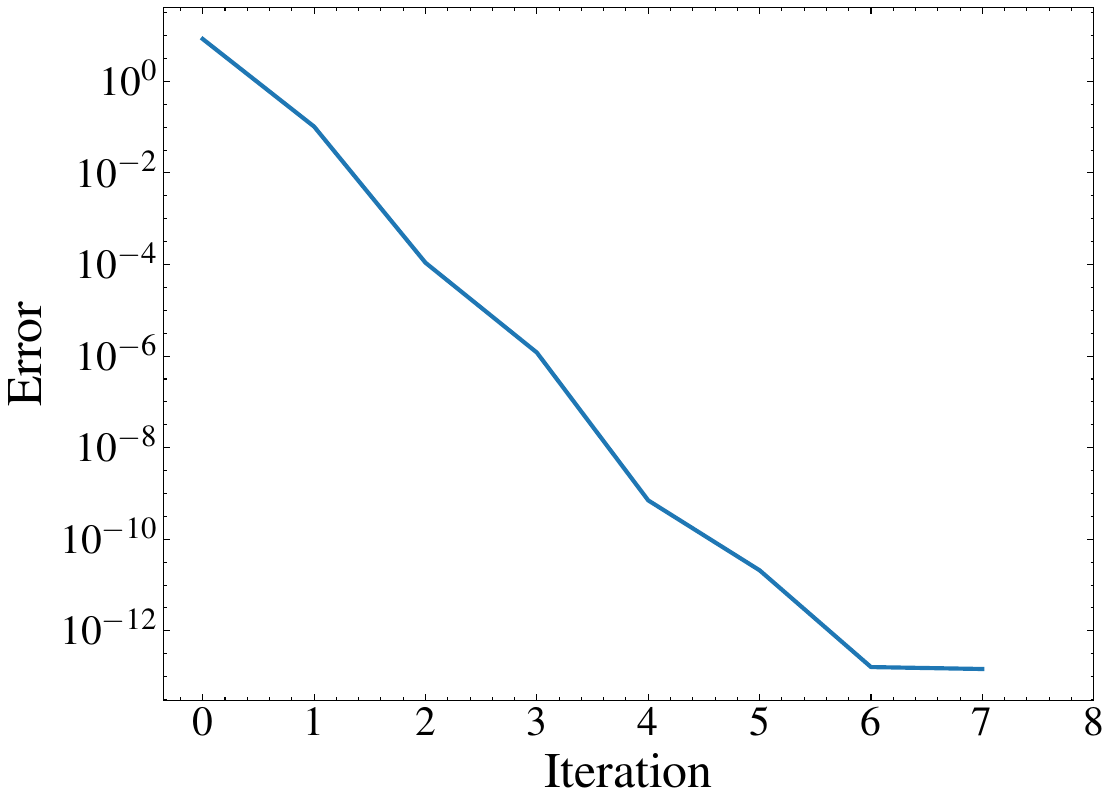}
    \label{fig:quadratic_func_plot_2048}
    \caption{\(N=2048\)}
  \end{subfigure}
  \caption{Approximation error \(\left\lVert {\bf x}^{(n)} - {\bf x}^* \right\rVert_2\) for the constrained quadratic optimization problem.}
  \label{fig:quadratic_func_plot}
\end{figure}

\subsection{Entropy loss} 
In the third simulation, we examine  the constrained optimization problem with the entropy loss function as the objective function.
%constrained by a set of linear equations and inequality constraints. 
In particular, we consider the following optimization problem,
\begin{align*}
  \min_{{\bf x}} \sum_{i\in V} x_{i} \log x_{i} \quad \text{subject to} \quad \mathbf{A} \mathbf{x} = \mathbf{b} \quad \text{and} \quad \mathbf{x} \geq \mathbf{0},
\end{align*}
where $\mathbf{A}$ is a given matrix, $\mathbf{b}$ is a given vector, and $\mathbf{x} \geq \mathbf{0}$ ensures nonnegativity of each component of 
 $\mathbf{x}=[x_i]_{i\in V}$. We point out that this formulation aligns with the entropy maximization framework in \cite{Jaynes1957}.

We set ${\bf A}=\chi_{W}(5{\bf L}_{\mathcal{G}} + {\bf I})$ and ${\bf b}$ as a random vector with each component uniformly distributed in $[0,1]$.   
We add a log barrier penalty term to the objective function with a parameter $t=100$ to convert the inequality constraint into an equality constraint. We then apply the proposed DAC algorithm to solve the optimization problem for the entropy loss function on the random geometric graph of order \(N=1024\) and \(N=2048\) . We compute the approximation error \(\lVert {\bf x}^{(n)} - {\bf x}^* \rVert_2\) at each iteration, where \({\bf x}^{(n)}\) is the solution obtained from the proposed DAC algorithm at the \(n\)-th iteration and \({\bf x}^*\) is the solution obtained from the centralized optimization solver using the interior-point method \cite{YuriiNesterov1994}. We display the average of the approximation error \(\lVert {\bf x}^{(n)} - {\bf x}^* \rVert_2\) of 100 trials in Figure \ref{fig:entropy_func_plot}. Similar to the previous two simulations, we observe that the proposed DAC algorithm converges to the optimal solution with an exponential rate.

\begin{figure}[t] %[!h]
  \begin{subfigure}[t]{0.48\textwidth}
    \includegraphics[width = \textwidth]{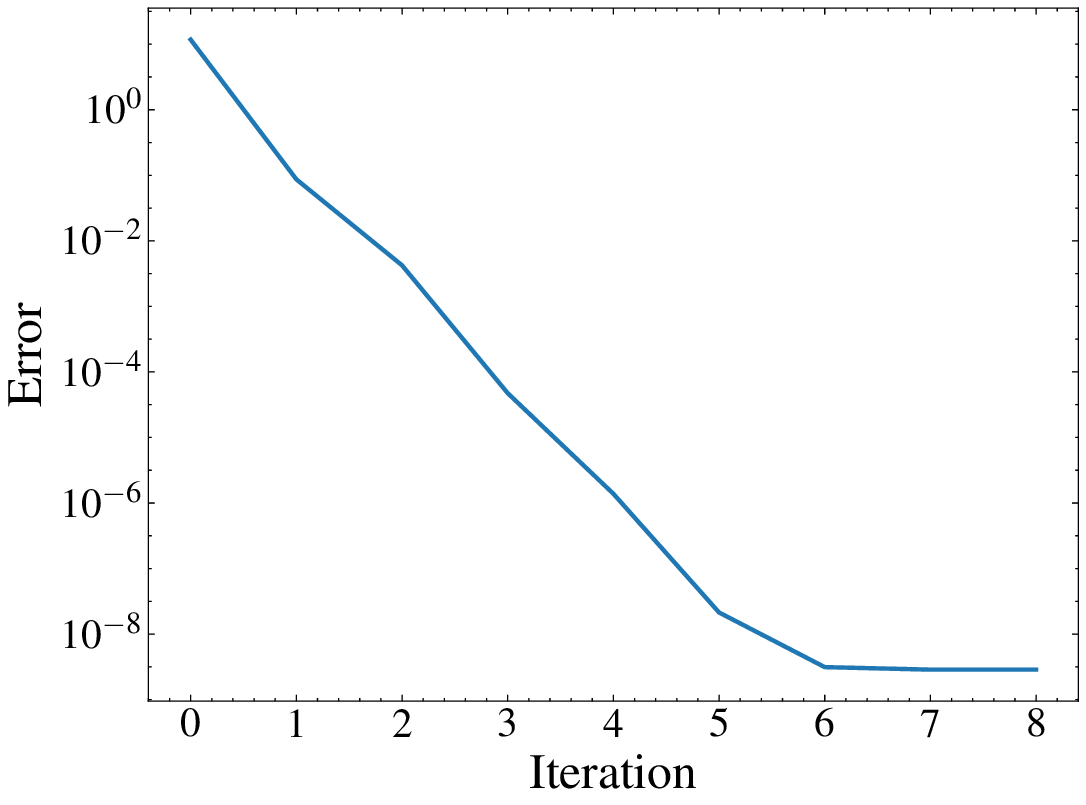}
    \label{fig:entropy_func_plot_1024}
    \caption{\(N=1024\)}
  \end{subfigure}
  \begin{subfigure}[t]{0.48\textwidth}
    \includegraphics[width = \textwidth]{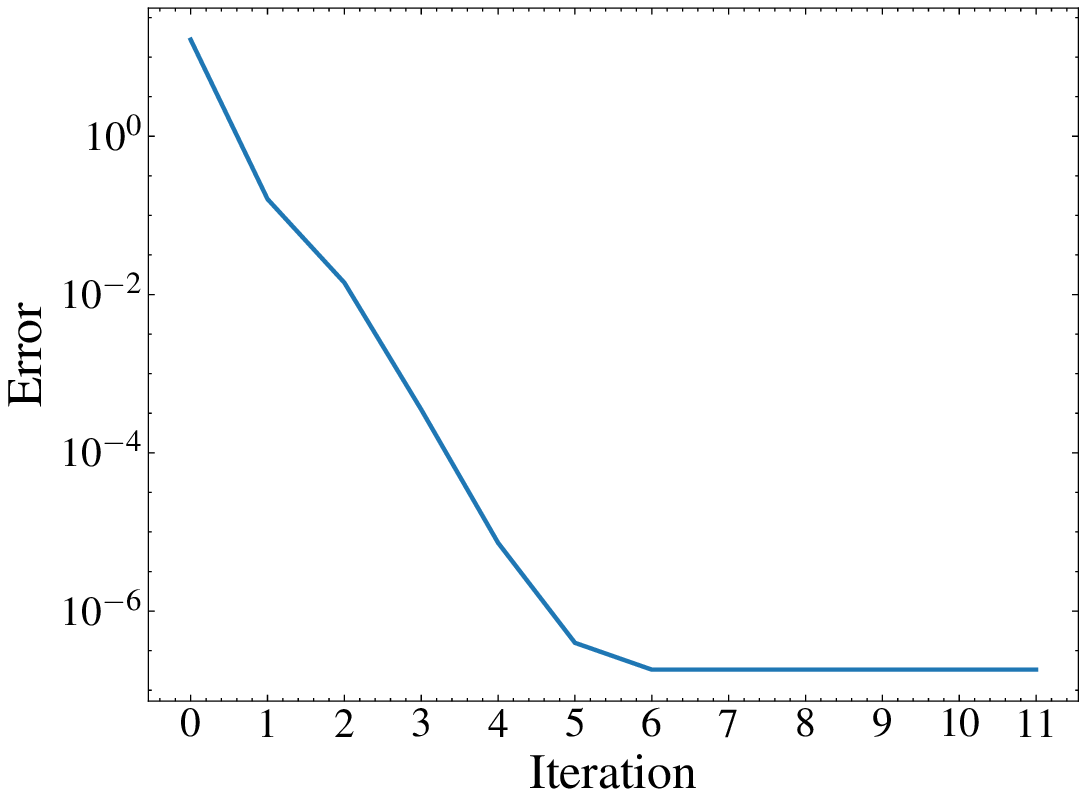}
    \label{fig:entropy_func_plot_2048}
    \caption{\(N=2048\)}
  \end{subfigure}
  \caption{Approximation error \(\lVert {\bf x}^{(n)} - {\bf x}^* \rVert_2\) for the entropy maximization problem.}
  \label{fig:entropy_func_plot}
\end{figure}

\section{Proofs}\label{proofs.section}

In this section, we collect proofs of Theorems \ref{maintheorem1.thm}, \ref{inexactmaintheorem1.thm} and \ref{thm:logbarrier}.

\subsection{Proof of Theorem \ref{maintheorem1.thm}}\label{maintheorem.thm.pfsection}
To prove  Theorem \ref{maintheorem1.thm}, we first establish  exponential off-diagonal decay property of the inverse of a  matrix with bounded geodesic width and block entries.

\begin{lemma}\label{maintheorem1.lem1}
  Let  matrices
  ${\bf H}_1=  [ h_1(i,j)]_{i,j\in V}$ and ${\bf H}_2=  [ h_2(i,j)]_{i,j\in V}$
  consist of  real entries $h_1(i,j), h_2(i,j)\in {\mathbb R}, i,j\in V$.
  If ${\bf H}_1$ is invertible and if ${\bf H}_1$ and ${\bf H}_2$
  have
  the geodesic width $\omega$, i.e.,
  \begin{equation}  \label{maintheorem1.lem1.eq1}
    h_1(i,j)=h_2(i,j)=0 \ \ {\rm for \ all}\ \  i, j\in V \ \ {\rm with}\ \ \rho(i,j)>\omega,
  \end{equation}
  then ${\bf H}_1^{-1}= [g_1(i,j)]_{i, j\in V}$
  and ${\bf H}_1^{-1}{\bf H}_2= [g_2(i,j)]_{i, j\in V}$
  have exponential off-diagonal decay,
  \begin{equation} \label{maintheorem1.lem1.eq2}
    |g_1(i,j)|\le \left\{\begin{array}{ll}
      \frac{(\kappa({\bf H}_1))^2}{\|{\bf H}_1\|} & { \rm if}\  j=i   \\
      \frac{ (\kappa({\bf H}_1))^2}{\|{\bf H}_1\|}
      \Big(\frac{(\kappa({\bf H}_1))^2-1}{(\kappa({\bf H}_1))^2+1}\Big)^{ \rho(i,j)/(2\omega)-1/2}
                                                  & {\rm if} \ j\ne i
    \end{array}\right.
  \end{equation}
  and
  \begin{equation} \label{maintheorem1.lem1.eq3}
    |g_2(i,j)|\le \left\{\begin{array}{ll}
      \frac{(\kappa({\bf H}_1))^2\|{\bf H}_2\|}{\|{\bf H}_1\|} & { \rm if}\  j=i    \\
      \frac{ (\kappa({\bf H}_1))^2 \|{\bf H}_2\|}{\|{\bf H}_1\|}
      \Big(\frac{(\kappa({\bf H}_1))^2-1}{(\kappa({\bf H}_1))^2+1}\Big)^{ \rho(i,j)/(2\omega)-1}
                                                               & {\rm if} \ j\ne i,
    \end{array}\right.
  \end{equation}
  where
  $\kappa({\bf H}_1)= \|{\bf H}_1^{-1}\| \|{\bf H}_1\|\ge 1$ is the condition number of the block  matrix ${\bf H}_1$.
\end{lemma}

The reader may refer to \cite{Emirov2022, Demko1984, Grochenig2006, Motee2017, Shin19, Sun2007, sunca11} for various off-diagonal decay properties for the inverse of infinite matrices.
For matrices with bounded geodesic width and scalar entries, the exponential off-diagonal decay property of  their inverses is discussed in \cite{Emirov2022}. We follow the  argument in \cite[Lemma 6.1]{Emirov2022} to prove Lemma \ref{maintheorem1.lem1} and we include a sketch of the proof for the completeness of this paper.

\begin{proof} [Proof of Lemma \ref{maintheorem1.lem1}]
For that case that $\omega=0$,  %${\bf H}_1$  and ${\bf H}_2$ are block diagonal matrices.
 ${\bf H}_1^{-1}$ and ${\bf H}_1^{-1} {\bf H}_2$ are block diagonal matrices and
norm of  their diagonal entries are dominated by their operator norms. This proves
 \eqref{maintheorem1.lem1.eq2}  and \eqref{maintheorem1.lem1.eq3} with $\omega=0$.

Now we consider the case that $\omega\ne 0$.
%it remains to prove \eqref{maintheorem1.lem1.eq2}  and \eqref{maintheorem1.lem1.eq3} for the case that $\omega\ne 0$.
% Let $L=\sum_{i\in V} L_i$.
 Then
  ${\bf H}_1^* {\bf H}_1$ is positive definite and
\begin{equation} \label{maintheorem1.lem1.pf.eq1}
\frac{\|{\bf H}_1\|^2}{ (\kappa({\bf H}_1))^2}
 {\bf I}_L\le {\bf H}_1^* {\bf H}_1\le \|{\bf H}_1\|^2 {\bf I}_L
\end{equation}
by the invertibility of ${\bf H}_1$.
Therefore
\begin{equation} \label{maintheorem1.lem1.pf.eq4}
\Big\|\Big( {\bf I}_L- \frac{2 (\kappa({\bf H}_1))^2}{\|{\bf H}_1\|^2 ((\kappa({\bf H}_1))^2+1)} {\bf H}_1^* {\bf H}_1\Big)^n \Big\|\le   \Big(\frac{(\kappa({\bf H}_1))^2-1}{(\kappa({\bf H}_1))^2+1}\Big)^n, \  n\ge 0,
\end{equation}
and
\begin{equation}  \label{maintheorem1.lem1.pf.eq2}
{\bf H}_1^{-1}= \frac{2 (\kappa({\bf H}_1))^2}{\|{\bf H}_1\|^2 ((\kappa({\bf H}_1))^2+1)}  \sum_{n=0}^\infty  \Big( {\bf I}_L- \frac{2 (\kappa({\bf H}_1))^2}{\|{\bf H}_1\|^2 ((\kappa({\bf H}_1))^2+1)}  {\bf H}_1^* {\bf H}_1\Big)^n {\bf H}_1^*.\end{equation}
By \eqref{maintheorem1.lem1.pf.eq4}  and  \eqref{maintheorem1.lem1.pf.eq2}, we have
\begin{eqnarray} \label{maintheorem1.lem1.pf.eq5}
|g_1(i,i)| &\hskip-0.08in  \le & \hskip-0.08in \frac{2 (\kappa({\bf H}_1))^2}{\|{\bf H}_1\|^2 ((\kappa({\bf H}))^2+1)}\nonumber\\
& & \times \sum_{n=0}^\infty \left\|\Big( {\bf I}_L- \frac{2 (\kappa({\bf H}_1))^2}{\|{\bf H}_1\|^2 ((\kappa({\bf H}_1))^2+1)}
 {\bf H}_1^* {\bf H}_1\Big)^n {\bf H}_1^*\right\|\nonumber\\
& \hskip-0.08in \le & \hskip-0.08in \frac{2 (\kappa({\bf H}_1))^2}{\|{\bf H}_1\| ((\kappa({\bf H}_1))^2+1)}
 \sum_{n=0}^\infty \Big(\frac{(\kappa({\bf H}_1))^2-1}{(\kappa({\bf H}_1))^2+1}\Big)^n=
\frac{ (\kappa({\bf H}_1))^2}{\|{\bf H}_1\|}\quad 
\end{eqnarray}
and
\begin{eqnarray} \label{maintheorem1.lem1.pf.eq6}
|g_2(i,i)| %& \le & \frac{2 (\kappa({\bf H}_1))^2}{\|{\bf H}_1\|^2 ((\kappa({\bf H}))^2+1)}   \sum_{n=0}^\infty \left\|\Big( {\bf I}_L- \frac{2 (\kappa({\bf H}_1))^2}{\|{\bf H}_1\|^2 ((\kappa({\bf H}_1))^2+1)}
% {\bf H}_1^* {\bf H}_1\Big)^n {\bf H}_1^*\right\|\nonumber\\
 & \hskip-0.08in  \le & \hskip-0.08in    \frac{2 (\kappa({\bf H}_1))^2\|{\bf H}_2\|}{\|{\bf H}_1\| ((\kappa({\bf H}_1))^2+1)}
 \sum_{n=0}^\infty \Big(\frac{(\kappa({\bf H}_1))^2-1}{(\kappa({\bf H}_1))^2+1}\Big)^n\nonumber\\
 & \hskip-0.08in  = & \hskip-0.08in 
\frac{ (\kappa({\bf H}_1))^2 \|{\bf H}_2\|}{\|{\bf H}_1\|},\ \ i\in V.
\end{eqnarray}

Let  $n_0(i,j)$ and $n_1(i,j)$ be the  smallest nonnegative integers such that
$(2n_0+1)\omega\ge \rho(i,j)$ and $(2n_1+2)\omega\ge \rho(i,j)$.
Observe that
\begin{equation} \label{maintheorem1.lem1.pf.eq3}
\omega\Big( \Big( {\bf I}_L- \frac{2 (\kappa({\bf H}_1))^2}{\|{\bf H}_1\|^2 ((\kappa({\bf H}_1))^2+1)}  {\bf H}_1^* {\bf H}_1\Big)^n \Big)\le 2n\omega, \ n\ge 0.
\end{equation}
Then for  distinct vertices $i,j\in V$, we obtain from
\eqref{maintheorem1.lem1.pf.eq4}, \eqref{maintheorem1.lem1.pf.eq2} and
\eqref{maintheorem1.lem1.pf.eq3} that
\begin{eqnarray}  \label{maintheorem1.lem1.pf.eq7}
|g_1(i,j)| %& \le &  \frac{2}{C_0^2+c_0^2}  \sum_{n=n_0(i,j)}^\infty \left\|\Big( {\bf I}_L- \frac{2}{C_0^2+c_0^2} {\bf H}^* {\bf H}\Big)^n {\bf H}^*\right\|\nonumber\\
& \hskip-0.08in \le &  \hskip-0.08in \frac{2 (\kappa({\bf H}_1))^2}{\|{\bf H}_1\| ((\kappa({\bf H}_1))^2+1)}
 \sum_{n=n_0(i,j)}^\infty \Big(\frac{(\kappa({\bf H}_1))^2-1}{(\kappa({\bf H}_1))^2+1}\Big)^n\nonumber\\
 & \hskip-0.08in = &  \hskip-0.08in
\frac{(\kappa({\bf H}_1))^2}{\|{\bf H}_1\|}
\Big(\frac{(\kappa({\bf H}_1))^2-1}{(\kappa({\bf H}_1))^2+1}\Big)^{n_0(i,j)}
\nonumber\\
& \hskip-0.08in \le &  \hskip-0.08in
\frac{ (\kappa({\bf H}_1))^2}{\|{\bf H}_1\|}
\Big(\frac{(\kappa({\bf H}_1))^2-1}{(\kappa({\bf H}_1))^2+1}\Big)^{ \rho(i,j)/(2\omega)-1/2},
\end{eqnarray}
and
\begin{eqnarray}  \label{maintheorem1.lem1.pf.eq8}
|g_2(i,j)| %& \le &  \frac{2}{C_0^2+c_0^2}  \sum_{n=n_0(i,j)}^\infty \left\|\Big( {\bf I}_L- \frac{2}{C_0^2+c_0^2} {\bf H}^* {\bf H}\Big)^n {\bf H}^*\right\|\nonumber\\
& \hskip-0.08in \le &  \hskip-0.08in  \frac{2 (\kappa({\bf H}_1))^2 \|{\bf H}_2\|}{\|{\bf H}_1\| ((\kappa({\bf H}_1))^2+1)}
 \sum_{n=n_1(i,j)}^\infty \Big(\frac{(\kappa({\bf H}_1))^2-1}{(\kappa({\bf H}_1))^2+1}\Big)^n
%\frac{(\kappa({\bf H}_1))^2}{\|{\bf H}_1\|}
%\Big(\frac{(\kappa({\bf H}_1))^2-1}{(\kappa({\bf H}_1))^2+1}\Big)^{n_0(i,j)}
\nonumber\\
& \hskip-0.08in \le &  \hskip-0.08in \frac{ (\kappa({\bf H}_1))^2\|{\bf H}_2\|}{\|{\bf H}_1\|}
\Big(\frac{(\kappa({\bf H}_1))^2-1}{(\kappa({\bf H}_1))^2+1}\Big)^{ \rho(i,j)/(2\omega)-1}.
\end{eqnarray}
Therefore the conclusions \eqref{maintheorem1.lem1.eq2}  and \eqref{maintheorem1.lem1.eq3} with $\omega\ne 0$
follow from \eqref{maintheorem1.lem1.pf.eq5}, \eqref{maintheorem1.lem1.pf.eq6}, \eqref{maintheorem1.lem1.pf.eq7} and \eqref{maintheorem1.lem1.pf.eq8}.
\end{proof}

\medskip

For $x, y\in {\mathbb R}^N$ and $\lambda\in \Lambda$,  define
\begin{equation}\label{maintheorem1.lem2.eq1a}
  {\bf K}_{1,\lambda}(x, y)=\left[\begin{matrix}
    \chi_{_{\DlR}} {\bf J}(x, y) \chi_{_{\DlR}}^* & \chi_{_{\DlR}} {\bf A}^T  \chi_{_{\Gl}}^* \\
    \chi_{_{\Gl}} {\bf A} \chi_{_{\DlR}}^*                                                  & {\bf 0}
  \end{matrix}\right]
\end{equation}
and
\begin{equation} \label{maintheorem1.lem2.eq1b}
  {\bf K}_{2, \lambda}(x, y)= \left[
  \begin{matrix}
    \chi_{_{\DlR}} {\bf J}(x, y) \\
    \chi_{_{\Gl}} {\bf A}
  \end{matrix}\right],
\end{equation}
where ${\bf J}(x, y)$ is given in Assumption \ref{assump:J}.
In the following lemma, we show that  ${\bf K}_{1, \lambda}(x, y)$ and
${\bf K}_{2, \lambda}(x, y)$ are uniformly bounded, and   ${\bf K}_{1, \lambda}(x, y)$ has bounded inverse.
% and their inverses have exponential off-diagonal decay.

\begin{lemma} \label{maintheorem1.lem2} Let ${\bf J}$ and ${\bf A}$  be as in Theorem \ref{maintheorem1.thm},
  and
  ${\bf K}_{1, \lambda}(x, y)$  and  ${\bf K}_{2, \lambda}(x, y)$
  be as in \eqref{maintheorem1.lem2.eq1a} and \eqref{maintheorem1.lem2.eq1b}.
  Then
  \begin{equation}\label {maintheorem1.lem2.eq3}
    \|{\bf K}_{1, \lambda}(x, y)\|\le  L_1+  \|{\bf A}\|,
  \end{equation}
  \begin{equation}\label {maintheorem1.lem2.eq4}
    \|({\bf K}_{1, \lambda}(x, y))^{-1}\|\le   \Big(\frac{c_1+\|{\bf A}\|}{c_1}\Big)^2 \max\Big( \frac{1}{c_1}, \frac{L_1}{c_2^2}\Big),
  \end{equation}
  and
  \begin{equation}\label {maintheorem1.lem2.eq5}
    \|{\bf K}_{2,\lambda}(x, y)\|\le  L_1+  \|{\bf A}\|
  \end{equation}
  hold for all $x, y\in {\mathbb R}^N$ and $\lambda\in \Lambda$,
  where
  $c_1, c_2$ and $L_1$ are constants in \eqref{assump:A} and Assumption \ref{assump:J}.
\end{lemma}

\begin{proof}
  Set ${\bf B}_\lambda=\chi_{_{\DlR}} {\bf J}(x, y) \chi_{_{\DlR}}^*$ and
  ${\bf A}_\lambda=\chi_{_{\Gl}} {\bf A} \chi_{_{\DlR}}^*$. We first prove \eqref{maintheorem1.lem2.eq3}. By a direct computation, we have
  \begin{eqnarray*}
    \|{\bf K}_{1, \lambda}(x, y)\|^2  %&= \max_{\|{\bf u}\|_2=1}\|\mK{\bf u}\|_2^2\\
    & \hskip-0.08in = & \hskip-0.08in  \sup_{\|{\bf u}_1\|_2^2 + \|{\bf u}_2\|_2^2=1} \left\|
    \left[\begin{matrix}
      {\bf B}_\lambda & {\bf A}_\lambda^T   \nonumber \\
      {\bf A}_\lambda
                      & {\bf 0}
    \end{matrix}\right]
    \left[\begin{matrix}
      {\bf u}_1 \\
      {\bf u}_2
    \end{matrix} \right] \right\|_2^2\\
    & \hskip-0.08in  = \hskip-0.08in  &\sup_{\|{\bf u}_1\|_2^2 + \|{\bf u}_2\|_2^2=1} \|{\bf B}_\lambda {\bf u}_1 +  {\bf A}_\lambda^T  {\bf u}_2\|_2^2 + \|{\bf A}_\lambda{\bf u}_1\|_2^2 \nonumber\\
    & \hskip-0.08in  = & \hskip-0.08in \sup_{\|{\bf u}_1\|_2^2 + \|{\bf u}_2\|_2^2=1} (L_1\|{\bf u}_1\|_2 + \|{\bf A}\| \|{\bf u}_2\|_2)^2 + \|{\bf A}\| ^2 \|{\bf u}_1\|_2^2
    \nonumber\\
    & \hskip-0.08in  \le  & \hskip-0.08in (L_1+\|{\bf A}\|)^2.
  \end{eqnarray*}

  We next show  that \eqref{maintheorem1.lem2.eq4}  holds. Set ${\bf C}_\lambda={\bf A}_\lambda ({\bf B}_\lambda)^{-1} {\bf A}_\lambda^T$.
  By Assumptions \ref{assump:J} and \ref{assump:A}, we obtain
  \begin{equation} \label{maintheorem1.lem2.pf.eq1}
    c_1 {\bf I}_{_{\DlR}} \preceq {\bf B}_\lambda \preceq L_1 {\bf I}_{_{\DlR}}
    \quad {\rm and} \quad
    \frac{c_2^2}{L_1} {\bf I}_{_{\Gl}} \preceq  {\bf C}_\lambda
    \preceq \frac{\|{\bf A}\|^2}{c_1} {\bf I}_{_{\Gl}}.
  \end{equation}
  Observe that
  \begin{eqnarray*}
    \big({\bf K}_{1,\lambda}(x, y)\big)^{-1}  &\hskip-0.08in   = & \hskip-0.08in
   \left[ \begin{matrix} {\bf I}_{_{\DlR}} & -  ({\bf B}_\lambda)^{-1}  {\bf A}_\lambda^T \\
                {\bf 0}
                                  & {\bf I}_{_{\Gl}}
    \end{matrix}\right]
  \left[  \begin{matrix} ({\bf B}_\lambda)^{-1} & {\bf 0}                   \\
                {\bf 0}
                                       & -  ({\bf C}_\lambda)^{-1}
    \end{matrix}\right]\nonumber\\
    & & \times
    \left[\begin{matrix} {\bf I}_{_{\DlR}} & {\bf 0}          \\
                -    {\bf A}_\lambda ({\bf B}_\lambda)^{-1}
                                  & {\bf I}_{_{\Gl}}
    \end{matrix}\right].
  \end{eqnarray*}
  This together with \eqref{maintheorem1.lem2.pf.eq1}
  and Assumption \ref{assump:A} proves \eqref{maintheorem1.lem2.eq4}.

  By direct computation, we get
  \begin{eqnarray*}
    \|{\bf K}_{2, \lambda}(x, y)\|^2  %&= \max_{\|{\bf u}\|_2=1}\|\mK{\bf u}\|_2^2\\
    & \hskip-0.08in  = & \hskip-0.08in \sup_{\|{\bf u}\|_2=1} \|\chi_{_{\DlR}} {\bf J}(x, y) {\bf u}\|^2+ \|\chi_{_{\Gl}} {\bf A}{\bf u}\|^2
    \nonumber\\
    & \hskip-0.08in  = & \hskip-0.08in  L_1^2 + \|{\bf A}\|^2 \le (L_1+\|{\bf A}\|)^2,
  \end{eqnarray*}
  which implies \eqref{maintheorem1.lem2.eq5} immediately.
\end{proof}

We next show that $({\bf K}_{1, \lambda}(x, y))^{-1}$  and
$({\bf K}_{1, \lambda}(x, y))^{-1} {\bf K}_{2, \lambda}(x', y')$
have exponential off-diagonal decay for all $x, x', y, y'\in {\mathbb R}^N$.

\begin{lemma} \label{maintheorem1.lem3} Let $m, \kappa, {\bf J}$ and ${\bf A}$  be as in Theorem \eqref{maintheorem1.thm}, and let
  ${\bf K}_{1, \lambda}(x, y)$ and
  ${\bf K}_{2, \lambda}(x, y)$
  be as in \eqref{maintheorem1.lem2.eq1a} and \eqref{maintheorem1.lem2.eq1b}.
  For $x, x', y, y'\in {\mathbb R}^N$, write
  $$\big({\bf K}_{1, \lambda}(x, y)\big)^{-1}=
    \left[\begin{matrix} \big({ g}_{11}(i,j)\big)_{i,j\in {\DlR}}          & \big({ g}_{12}(i,j)\big)_{i\in {\DlR}, j\in {\Gl}} \\
                \big({ g}_{21}(i,j)\big)_{i\in {\Gl},j\in {\DlR}} & \big({g}_{22}(i,j)\big)_{i, j\in {\Gl}}
    \end{matrix}\right]
  $$
  and
  $$\big({\bf K}_{1, \lambda}(x, y)\big)^{-1}{\bf K}_{2, \lambda}(x', y')=
  \left[ \begin{matrix} \big({ g}_{13}(i,j)\big)_{i\in {\DlR}, j\in V} \\
      \big({ g}_{23}(i,j)\big)_{i\in {\Gl},j\in V}
    \end{matrix}\right].
  $$
  Then for $l, l'\in \{1,2\} $,
  \begin{equation} \label{maintheorem1.lem3.eq1}
    |g_{_{ll'}}(i,j)|\le \left\{\begin{array}{ll}
      \frac{\kappa^2}{L_1+  \|{\bf A}\|} & { \rm if}\  j=i   \\
      \frac{\kappa^2}{L_1+  \|{\bf A}\|}
      \Big(\frac{\kappa^2-1}{\kappa^2+1}\Big)^{ \rho(i,j)/(4m)-1/2}
                                         & {\rm if} \ j\ne i
    \end{array}\right.
  \end{equation}
  hold for vertices  $i, j$ in  the appropriate index sets being either ${\DlR}$ or $\Gl$,
  and for $(l, l')=(1, 3)$ or $(2, 3)$,
  \begin{equation} \label{maintheorem1.lem3.eq2}
    |g_{_{ll'}}(i,j)|\le \left\{\begin{array}{ll}
      \kappa^2 & { \rm if}\  j=i   \\
      \kappa^2
      \Big(\frac{\kappa^2-1}{\kappa^2+1}\Big)^{ \rho(i,j)/(4m)-1}
               & {\rm if} \ j\ne i
    \end{array}\right.
  \end{equation}
  hold for vertices $j\in V$ and  $i$ in  the appropriate index sets being either ${\DlR}$ or $\Gl$.
\end{lemma}

\begin{proof}
  Take  $\mu\in [(c_1/(c_1+\|{\bf A}\|))^2 \min ( c_1, c_2^2/L_1),
    L_1+\|{\bf A}\|]$,
  define
  \begin{equation} \label{maintheorem1.lem3.pf.eq1}
    {\bf H}_{1, \lambda}=
   \left[ \begin{matrix} \mu {\bf I}_{V\backslash (W_{\lambda, R}\cup D_{\lambda,R})} & {\bf 0}                   \\
                {\bf 0}                                                      & {\bf K}_{1, \lambda}(x,y)
    \end{matrix}\right]= \big(h_{1, \lambda}(i,j)\big)_{i, j\in V}.
  \end{equation}
  By \eqref{maintheorem1.lem3.pf.eq1} and Lemma \eqref{maintheorem1.lem2}, we have
  \begin{equation}\label {maintheorem1.lem3.eq3}
    \|{\bf H}_{1, \lambda}\|\le  L_1+  \|{\bf A}\|, \ \
    \|({\bf H}_{1, \lambda})^{-1}\|\le   \Big(\frac{c_1+\|{\bf A}\|}{c_1}\Big)^2 \max\Big( \frac{1}{c_1}, \frac{L_1}{c_2^2}\Big)
  \end{equation}
  and
  \begin{equation} \label {maintheorem1.lem3.eq4}
    ({\bf H}_{1, \lambda})^{-1}=
    \left[\begin{matrix} \mu^{-1} {\bf I}_{V\backslash (W_{\lambda, R}\cup D_{\lambda,R})} & {\bf 0}                           \\
                {\bf 0}                                                           & ({\bf K}_{1, \lambda}(x, y))^{-1}
    \end{matrix}\right].
  \end{equation}
  Reorganizing the index set in the definition of the  matrix ${\bf H}_{1, \lambda}$, we can consider that
  ${\bf H}_{1, \lambda}$
  consists of blocks $h_{1, \lambda}(i,j)$
  and it has geodesic width $2m$ by Assumptions
  \eqref{assump:J} and \eqref{assump:A}.
  This together with \eqref{maintheorem1.lem3.eq3}, \eqref{maintheorem1.lem3.eq4}
  and Lemma \ref{maintheorem1.lem1} proves \eqref{maintheorem1.lem3.eq1}.

  Define
  \begin{equation*}
    {\bf H}_{2, \lambda}=
    \left[\begin{matrix}  {\bf 0} \\
      {\bf K}_{\lambda}^2(x', y')
    \end{matrix}\right]= \big(h_{2, \lambda}(i,j)\big)_{i, j\in V}.
  \end{equation*}
  Then
  \begin{equation} \label {maintheorem1.lem3.eq5}
    \|{\bf H}_{2, \lambda}\|\le  L_1+  \|{\bf A}\|
    \ \ {\rm and} \ \  ({\bf H}_{1,  \lambda})^{-1} {\bf H}_{2, \lambda}= \
   \left[\begin{matrix}  {\bf 0} \\
      ({\bf K}_{1,\lambda}(x, y))^{-1} {\bf K}_{2, \lambda}(x', y')
    \end{matrix}\right].
  \end{equation}
  Reorganizing the index set in the definition of the  matrix ${\bf H}_{2, \lambda}$, we can consider that
  ${\bf H}_{2, \lambda}$
  consists of blocks $h_{2, \lambda}(i,j)\in {\mathbb R}$
  and it has geodesic width $2m$ by Assumptions
  \eqref{assump:J} and \eqref{assump:A}.
  This together with \eqref{maintheorem1.lem3.eq3}, \eqref{maintheorem1.lem3.eq4}, \eqref{maintheorem1.lem3.eq5}
  and Lemma \ref{maintheorem1.lem1} proves \eqref{maintheorem1.lem3.eq2}.
\end{proof}

Now we are ready to  prove Theorem
\ref{maintheorem1.thm}.

\begin{proof} [Proof of Theorem \ref{maintheorem1.thm}]
  Let $\widetilde S$  and  $\widetilde S_{\lambda, R}$ be the cardinalities of the sets $W$ and $W_{\lambda, R}$ respectively.
  %  \sum_{i\in W} \widetilde s_i$ and
  %$\widetilde S_{\lambda, R}=\sum_{i\in W_{\lambda, R}}\widetilde s_i, \lambda \in \Lambda$,
  Set ${\bf x}_{\lambda, R}^*=\chi_{D_{\lambda, R}}^*
    \chi_{_{D_{\lambda, R}}}   {\bf x}^*$ and ${\bf x}_{\lambda, R}^{(n)}=\chi_{D_{\lambda, R}}^*
    \chi_{_{D_{\lambda, R}}}   {\bf x}^{(n)}, n\ge 0$.
  By the Karush–Kuhn–Tucker conditions for the constrained optimization problem \eqref{convexoptimizationlinearconstraint.def}, we can find a multiplier ${{\bf v}}^*=[{\bf v}_i]_{i\in W}\in {\mathbb R}^{\widetilde S}$ such that
  \begin{equation}\label{maintheorem1.pf.eq1}
    \nabla F({\bf x}^*)+{\bf A}^T {\bf v}^* = {\bf 0} \  \ {\rm and} \  \
    {\bf A} {\bf x}^* = {\bf b}.
  \end{equation}
  Similarly
  by the Karush–Kuhn–Tucker conditions for the local linearly constrained optimization
  \eqref{DACwithoutinequality.defb},
  there exist multipliers
  ${\bf v}_\lambda^{(n)}=[{\bf v}_{i}^{(n)}]_{i\in W_{\lambda, R}}\in {\mathbb R}^{\widetilde S_{\lambda, R}}, \lambda\in \Lambda$
  such that
  \begin{subequations}\label{maintheorem1.pf.eq2}
    \begin{equation}
      \label{maintheorem1.pf.eq2a}
      \chi_{_{\DlR}} \nabla F\big(\chi^*_{_{\DlR}} {\bf w}_\lambda^{(n)} -{\bf x}_{\lambda, R}^{(n)}+{\bf x}^{(n)}
      %+ \chi_{V\backslash D_{\lambda, R}}^* \chi_{V\backslash D_{\lambda, R}} {\bf x}^{(n)}
      \big)+\chi_{_{\DlR}}{\bf A}^T\chi^*_{_{W_{\lambda, R}}} {\bf v}^{(n)}_{\lambda} = {\bf 0} \end{equation}
    and
    \begin{equation}
      \label{maintheorem1.pf.eq2b}
      \chi_{_{W_{\lambda, R}}} {\bf A} \big(\chi^*_{_{\DlR}} {\bf w}_\lambda^{(n)} -{\bf x}_{\lambda, R}^{(n)}+{\bf x}^{(n)}
      %+ \chi_{V\backslash D_{\lambda, R}}^* \chi_{V\backslash D_{\lambda, R}}  {\bf x}^{(n)}
      \big) = \chi_{_{W_{\lambda, R}}} {\bf b}.
    \end{equation}
  \end{subequations}
  Combining
  \eqref{maintheorem1.pf.eq1} and \eqref{maintheorem1.pf.eq2} yields
  \begin{equation}\label{maintheorem1.pf.eq3}
    \chi_{_{\DlR}}\Big[\nabla F\big(\chi^*_{_{\DlR}} {\bf w}_\lambda^{(n)} -{\bf x}_{\lambda, R}^{(n)}+{\bf x}^{(n)}
    %+ \chi_{V\backslash D_{\lambda, R}}^* \chi_{V\backslash D_{\lambda, R}}  {\bf x}^{(n)}
    \big) - \nabla F({\bf x}^*) \Big] %\nonumber\\
    %& \hskip-0.08in  = & \hskip-0.08in
    =\chi_{_{\DlR}} {\bf A}^T ({\bf v}^* - \chi^*_{_{\Gl}} {\bf v}^{(n)}_{\lambda} )  %, \ \ \lambda\in \Lambda.
  \end{equation}
  and
  \begin{equation}  \label{maintheorem1.pf.eq3+}
    \chi_{_{W_{\lambda, R}}} {\bf A}\big(\chi^*_{_{\DlR}} {\bf w}_\lambda^{(n)}-{\bf x}^*)= -\chi_{_{W_{\lambda, R}}} {\bf A}
    (-{\bf x}_{\lambda, R}^{(n)}+{\bf x}^{(n)}),  %\chi_{V\backslash D_{\lambda, R}}^*
    %\chi_{V\backslash D_{\lambda, R}}  {\bf x}^{(n)}\big),
    \ \lambda\in \Lambda.
  \end{equation}

  Define
  $${\bf J}_1\big({\bf w}_\lambda^{(n)} , {\bf x}^{(n)}\big) = {\bf J}\big(\chi^*_{_{\DlR}} {\bf w}_\lambda^{(n)}
    -{\bf x}_{\lambda, R}^{(n)}+{\bf x}^{(n)}, % \chi_{V\backslash D_{\lambda, R}}^*\chi_{V\backslash D_{\lambda, R}}  {\bf x}^{(n)},
    \ {\bf x}_{\lambda, R}^*   -{\bf x}_{\lambda, R}^{(n)}+{\bf x}^{(n)} \big)$$
  % \chi_{V\backslash D_{\lambda, R}}^*
  %\chi_{V\backslash D_{\lambda, R}}  {\bf x}^{(n)}\big)$$
  and
  $${\bf J}_2({\bf x}^{(n)})={\bf J}\big( {\bf x}_{\lambda, R}^*
    + \chi_{V\backslash D_{\lambda, R}}^*
    \chi_{V\backslash D_{\lambda, R}}  {\bf x}^{(n)}, {\bf x}^*\big),$$
  where ${\bf x}_{\lambda, R}^*=\chi_{D_{\lambda, R}}^*
    \chi_{D_{\lambda, R}}   {\bf x}^*$ and ${\bf J}({\bf x}, {\bf y})$ is given  in
  \eqref{eq:gradient_difference}.
  Then
  \begin{eqnarray} \label{maintheorem1.pf.eq4}
    \hskip-0.08in& \hskip-0.08in & \hskip-0.08in \nabla F(\chi^*_{_{\DlR}} {\bf w}_\lambda^{(n)} + \chi_{V\backslash D_{\lambda, R}}^*
    \chi_{V\backslash D_{\lambda, R}}  {\bf x}^{(n)}) - \nabla F({\bf x}_{\lambda, R}^*
    + \chi_{V\backslash D_{\lambda, R}}^*
    \chi_{V\backslash D_{\lambda, R}}  {\bf x}^{(n)})\nonumber \\
    \hskip-0.08in& \hskip-0.08in  & \hskip-0.08in = {\bf J}_1\big({\bf w}_\lambda^{(n)} , {\bf x}^{(n)}\big) \chi^*_{_{\DlR}} \big({\bf w}_\lambda^{(n)}  - \chi_{_{\DlR}} {\bf x}^*\big)
  \end{eqnarray}
  and
  \begin{equation} \label{maintheorem1.pf.eq5}
    \nabla F({\bf x}_{\lambda, R}^*
    + \chi_{V\backslash D_{\lambda, R}}^*
    \chi_{V\backslash D_{\lambda, R}}  {\bf x}^{(n)})  - \nabla F({\bf x}^*) = {\bf J}_2({\bf x}^{(n)}) {\bf I}_{\DlRc} ({\bf x}^{(n)} -{\bf x}^*)
  \end{equation}
  by Assumption  \ref{assump:J}.
  Substituting the expressions in \eqref{maintheorem1.pf.eq4} and
  \eqref{maintheorem1.pf.eq5}
  into \eqref{maintheorem1.pf.eq3}, we obtain
  \begin{eqnarray}  \label{maintheorem1.pf.eq6}
    \hskip-0.08in  & \hskip-0.08in & \hskip-0.08in  \chi_{_{\DlR}} {\bf J}_1\big({\bf w}_\lambda^{(n)} , {\bf x}^{(n)}\big) \chi^*_{_{\DlR}}\big ({\bf w}_\lambda^{(n)}  - \chi_{_{\DlR}} {\bf x}^*\big) - \chi_{_{\DlR}} {\bf A}^T \big({\bf v}^* - \chi^*_{_{\Gl}} {\bf v}^{(n)}_{\lambda}\big )\nonumber\\
    & \hskip-0.08in  = & \hskip-0.08in  - \chi_{_{\DlR}} {\bf J}_2({\bf x}^{(n)}) \chi_{V\backslash D_{\lambda, R}}^*
    \chi_{_{V\backslash D_{\lambda, R}}} ({\bf x}^{(n)} -{\bf x}^*).
  \end{eqnarray}
  By the definition of the set $\Gl$ in \eqref{eq:Gamma_lambda}, we have
  \begin{equation} \label{maintheorem1.pf.eq7} \chi_{_{V\backslash {W_{\lambda, R}}}}{\bf A}\chi^*_{_{\DlR}}=0,\end{equation}
  which implies that
  $\chi_{\DlR} {\bf A}^T {\bf v}^* =  {\bf A}_\lambda^T  \chi_{_{W_{\lambda, R}}}{\bf v}^*$,
  where  ${\bf A}_{\lambda} = \chi_{_{\Gl}} {\bf A}\chi^*_{\DlR}$.
  This together with \eqref{maintheorem1.pf.eq6} yields %to the first crucial equation in our argument,
  \begin{eqnarray} \label{maintheorem1.pf.eq8}
    & & \hskip-0.08in {\bf B}^{(n)}_{\lambda} ({\bf w}_\lambda^{(n)}  - \chi_{_{\DlR}} {\bf x}^*) + {\bf A}_{\lambda}^T ({\bf v}^{(n)}_{\lambda} - \chi_{_{\Gl}} {\bf v}^*)\nonumber\\
    &  = & \hskip-0.08in  {\bf G}^{(n)}_{\lambda}
    \chi_{_{V\backslash D_{\lambda, R}}}^*  \chi_{_{V\backslash D_{\lambda, R}}}
    ({\bf x}^{(n)} -{\bf x}^*),
  \end{eqnarray}
  where
  %\begin{equation} \label{maintheorem1.pf.eq9}
  $\mB^{(n)}_{\lambda} = \chi_{_{\DlR}}{\bf J}_1\big({\bf w}_\lambda^{(n)} , {\bf x}^{(n)}\big) \chi^*_{_{\DlR}}$
  and
  ${\bf G}^{(n)}_\lambda = - \chi_{_{\DlR}} {\bf J}_2({\bf x}^{(n)})$.
  %\end{equation}

  Set $\widetilde {\bf A}_{\lambda} = - \chi_{_{\Gl}} {\bf A} $.
  By \eqref{maintheorem1.pf.eq1}, \eqref{maintheorem1.pf.eq2b} and \eqref{maintheorem1.pf.eq7},  we have
  \begin{eqnarray} \label{maintheorem1.pf.eq10}
    {\bf A}_{\lambda} ({\bf w}_\lambda^{(n)}  - \chi_{_{\DlR}} {\bf x}^*) & \hskip-0.08in  = & \hskip-0.08in  \chi_{_{\Gl}}{\bf A} (\chi^*_{_{\DlR}} {\bf w}_\lambda^{(n)} -{\bf x}^*)+
    \chi_{_{\Gl}}{\bf A} \chi_{_{V\backslash D_{\lambda, R}}}^*  \chi_{_{V\backslash D_{\lambda, R}}}  {\bf x}^*\nonumber \\
    & \hskip-0.08in = & \hskip-0.08in
    \widetilde {\bf A}_{\lambda} \chi_{_{V\backslash D_{\lambda, R}}}^*  \chi_{_{V\backslash D_{\lambda, R}}}  ({\bf x}^{(n)} - {\bf x}^*).
  \end{eqnarray}
  Combining \eqref{maintheorem1.pf.eq8}  and \eqref{maintheorem1.pf.eq10} yields the following crucial linear system,
  \begin{equation} \label{maintheorem1.pf.eq10}
  \left[  \begin{matrix}
      {\bf B}^{(n)}_{\lambda} & {\bf A}^T_{\lambda} \\
      {\bf A}_{\lambda}       & {\bf 0}
    \end{matrix}\right]
    \left[\begin{matrix}
      {\bf w}_\lambda^{(n)}  - \chi_{_{\DlR}} {\bf x}^* \\
      {\bf v}^{(n)}_{\lambda} - \chi_{_{\Gl}} {\bf v}^*
    \end{matrix}\right]
    =
    \left[\begin{matrix}
      \mG^{(n)}_{\lambda} \\
      \widetilde {\bf A}_{\lambda}
    \end{matrix} \right]{\bf I}_{\DlRc}
    ({\bf x}^{(n)} - {\bf x}^*).
  \end{equation}

  By Lemma \ref{maintheorem1.lem2}, we have
  \begin{equation} \label{maintheorem1.pf.eq11}
 \left\| \left[  \begin{matrix}
      {\bf B}^{(n)}_{\lambda} & {\bf A}^T_{\lambda} \\
      {\bf A}_{\lambda}       & {\bf 0}
    \end{matrix}\right]
\right\|\le L_1+\|{\bf A}\|,
  \end{equation}
  \begin{equation}
    \left\|\left[\begin{matrix}
      {\bf B}^{(n)}_{\lambda} & {\bf A}^T_{\lambda} \\
      {\bf A}_{\lambda}       & {\bf 0}
    \end{matrix}\right]^{-1}\right\|
    \le
    \Big(\frac{c_1+\|{\bf A}\|}{c_1}\Big)^2 \max\Big( \frac{1}{c_1}, \frac{L_1}{c_2^2}\Big),
  \end{equation}
  and
  \begin{equation} \label{maintheorem1.pf.eq12}
    \left\|\left[\begin{matrix}
      \mG^{(n)}_{\lambda} \\
      \widetilde {\bf A}_{\lambda}
    \end{matrix}\right]\right\|\le  L_1+  \|{\bf A}\|.
  \end{equation}
  Therefore by Lemma \ref{maintheorem1.lem3}, there exists a matrix
  ${\bf H}_\lambda^{(n)}=[h^{(n)}_{\lambda}(i,j)]_{i\in {\DlR}, j\in V}$ such that
  \begin{equation}  \label{maintheorem1.pf.eq13}
    {\bf w}_\lambda^{(n)}  - \chi_{\DlR} {\bf x}^* = {\bf H}_{\lambda}^{(n)} \chi_{_{V\backslash D_{\lambda, R}}}^*  \chi_{_{V\backslash D_{\lambda, R}}}   ({\bf x}^{(n)} - {\bf x}^*)
  \end{equation}
  and
  \begin{equation}  \label{maintheorem1.pf.eq14}
    |h^{(n)}_{\lambda}(i,j)|\le
    \left\{\begin{array}{ll}
      \kappa^2 & { \rm if}\  j=i   \\
      \kappa^2
      \Big(\frac{\kappa^2-1}{\kappa^2+1}\Big)^{ \rho(i,j)/(4m)-1}
               & {\rm if} \ j\ne i
    \end{array}\right.
  \end{equation}
  where $i\in {\DlR}, j\in V$.

  Define
  \begin{equation}
    {\bf H}^{(n)} = \sum_{\lambda\in \Lambda}\chi_{_{D_{\lambda}}}^*  \chi_{_{D_{\lambda}}}  \chi^*_{\DlR}{\bf H}_{\lambda}^{(n)} = [h^{(n)}(i,j)]_{i,j\in V}.
  \end{equation}
  By \eqref{extendedneighbor.def}, \eqref{DACwithoutinequality.defb}, \eqref{maintheorem1.pf.eq13}, \eqref{maintheorem1.pf.eq14}
  and non-overlapping covering property of $\Dl, \lambda\in \Lambda$, we obtain
  \begin{equation} \label{maintheorem1.pf.eq15}
    {\bf x}^{(n+1)} - {\bf x}^*  % = \sum_{\lambda\in \Lambda} \mI_{\Dl} \chi^*_{\DlR} ({\bf w}_\lambda^{(n)}  - \chi_{\DlR} {\bf x}^*)
    = {\bf H}^{(n)} ({\bf x}^{(n)} - {\bf x}^*),
  \end{equation}
  and
  \begin{equation}  \label{maintheorem1.pf.eq16}
    |h^{(n)}(i,j)|\le
    \left\{\begin{array}{ll}
      0 & { \rm if}\  \rho(j,i)\le R \\
      \kappa^2
      \Big(\frac{\kappa^2-1}{\kappa^2+1}\Big)^{ \rho(i,j)/(4m)-1}
        & {\rm if} \ \rho(j,i)> R
    \end{array}\right.
  \end{equation}
  where $i,j\in V$.
  Therefore following the same argument used in \cite{Emirov2022}, we have
  \begin{align*}
    \|{\bf x}^{(n+1)} - {\bf x}^*\|_p \le \delta_R
    \|{\bf x}^{(n)} - {\bf x}^*\|_p, \ n\ge 0.
  \end{align*}
  The desired  estimate in \eqref{maintheorem1.thm.eq4} follows from applying the above inequality iteratively.
\end{proof}

\subsection{Proof of Theorem \ref{inexactmaintheorem1.thm}}
\label{inaxactmaintheorem.thm.pfsection} 
Let 
$\widetilde {\bf w}_\lambda^{(n)}$ and $\widetilde {\bf x}^{(n)}$
be as in \eqref{eq:tldxn+1}, and set
   $$\widetilde{\bf K}_{\lambda}^{(n)} = \left[\begin{matrix}                   \widetilde{\mB}^{(n)}_{\lambda} & {\bf A}^T_{\lambda} \\
{\bf A}_{\lambda}   & \mzero                                    \end{matrix}\right], \  \lambda\in \Lambda, $$
where 
 ${\bf A}_{\lambda} = \chi_{_{\Gl}} {\bf A}\chi^*_{\DlR}$ and
  $  \widetilde{\mB}^{(n)}_{\lambda} =\chi_{\DlR} {\bf J}\big(\chi^*_{\DlR} \widetilde{{\bf w}}_\lambda^{(n)} + \mI_{\DlRc} \widetilde{{\bf x}}^{(n)}, \mI_{\DlR} {\bf x}^* + \mI_{\DlRc} \widetilde{{\bf x}}^{(n)}\big) \chi^*_{\DlR}$.  
By Lemma \ref{maintheorem1.lem2}, we have
\begin{equation} \label{inexactmaintheorem1.pf.eq11}
 \big \| \widetilde{\bf K}_{\lambda}^{(n)}
\big\|\le L_1+\|{\bf A}\| \ \ {\rm and}\ \
   \big\| \big(\widetilde{\bf K}_{\lambda}^{(n)}\big)^{-1}
\big\|\le 
  \Big(\frac{c_1+\|{\bf A}\|}{c_1}\Big)^2 \max\Big( \frac{1}{c_1}, \frac{L_1}{c_2^2}\Big).
\end{equation}

 Write   $$\big(\widetilde{\bf K}_{\lambda}^{(n)}\big)^{-1} =  \left[\Big[\big(\widetilde{\bf K}_{\lambda}^{(n)}\big)^{-1}\Big]_{ij} \right]_{1\le i, j\le 2} , \  \lambda\in \Lambda$$
   and define 
$$\widetilde{\bf H}^{(n)} = \sum_{\lambda\in \Lambda} \mI_{\Dl} \chi^*_{\DlR} \left(\left[\Big(\widetilde{\bf K}_{\lambda}^{(n)}\Big)^{-1}\right]_{11}\widetilde{\bf G}_{\lambda}^{(n)} +  \left[\Big(\widetilde{\bf K}_{\lambda}^{(n)}\Big)^{-1}\right]_{12}\widetilde{\bf A}_{\lambda}\right),$$
where
 $\widetilde {\bf A}_{\lambda} = - \chi_{_{\Gl}} {\bf A} $ and
$\widetilde{\bf G}^{(n)}_\lambda = - \chi_{\DlR} {\bf J}(\chi^*_{\DlR} \widetilde{{\bf w}}_\lambda^{(n)} + \mI_{\DlRc} \widetilde{{\bf x}}^{(n)}, {\bf x}^*)  \mI_{\DlRc}$.
Following the argument used to establish \eqref{maintheorem1.pf.eq16}, we obtain
$$
  \|\widetilde{\bf H}^{(n)}\|_{\cS} \le \delta_R
$$
and 
\begin{align*}
\left|\left[\big(\widetilde{\bf K}_{\lambda}^{(n)}\big)^{-1}\right]_{i,j}(k,l)\right| \le \frac{1}{\sqrt{L}}  \left(1 - c \right)^{\frac{\rho(k,l)}{4m}-1}
\ \ {\rm for \ all} \ \ 1\le i, j\le 2 \ \ {\rm and }\ k,l\in V.
\end{align*}
Here $\|{\bf B}\|_{\mathcal S}=\max\big(\sup_{i\in V}\sum_{i\in V} |b(i,j)|, \sup_{j\in V}\sum_{i\in V}|b(i,j)|\big)$ is the Schur norm of a matrix ${\bf B}=[b(i,j)]_{i,j\in V}$.

Let ${\pmb \theta}_\lambda^{(n)}$ and
${\pmb \eta}_{\lambda}^{(n)}$ be as in 
\eqref{inexactmaintheorem1.pf.eq2}, and 
${\bf x}^*$  be as in \eqref{maintheorem1.pf.eq1}.
Similar to the derivation of \eqref{maintheorem1.pf.eq15}, we could obtain the following relation between the approximation errors of $\tilde{{\bf x}}^{(n+1)}$ and $\tilde{{\bf x}}^{(n)}$:
\begin{eqnarray*}
  \tilde{{\bf x}}^{(n+1)} - {\bf x}^* &  = &  \widetilde{\bf H}^{(n)} (\tilde{{\bf x}}^{(n)} - {\bf x}^*)\nonumber\\
  & & + \sum_{\lambda\in \Lambda}\mI_{\Dl} \chi^*_{\DlR} \left(\left[\widetilde{\bf K}_{\lambda}^{(n)}\right]^{-1}_{11}\vtheta_{\lambda}^{(n)} +  \left[\widetilde{\bf K}_{\lambda}^{(n)}\right]^{-1}_{12}\veta_{\lambda}^{(n)}\right).
\end{eqnarray*}
Therefore 
\begin{align*}
  \|\tilde{{\bf x}}^{(n+1)} - {\bf x}^*\|_\infty \le \delta_R \|\tilde{{\bf x}}^{(n)} - {\bf x}^*\|_\infty + \frac{2}{\sqrt{L}}\delta_R \epsilon_n, \quad n\ge 0,
\end{align*}
and the desired result in \eqref{inexactmaintheorem1.thm.eq1} follows from applying the above inequality iteratively.

\subsection{Proof of Theorem \ref{thm:logbarrier}}
\label{thm:logbarrier.pfsection}

It follows immediately from Theorem \ref{maintheorem1.thm} by noting that the Hessian of the new objective function \(G\) in the log-barrier model \cref{eq:globalopt_logbarrier} is bounded by \(L_{1}+M_{t}\).

\bibliographystyle{plain}

\bibliography{ConstrainedDistributedOptimizationArXiv}

\end{document}